\documentclass[a4paper,reqno]{amsart}

\makeatletter
\@namedef{subjclassname@2020}{%
  \textup{2020} Mathematics Subject Classification}
\makeatother

\usepackage{amsmath}
\usepackage{amssymb}
\usepackage{amsthm}
\usepackage{mathrsfs}
\usepackage{latexsym}
\usepackage{amscd}
\usepackage{xypic}
\xyoption{curve}
\usepackage{ifthen} 
\usepackage{hyperref} 
\usepackage{graphicx}
\usepackage{enumerate}
\usepackage{enumitem}
\usepackage{rotating}
\usepackage{xcolor}
\usepackage{mathtools}
\usepackage{todonotes}
\usepackage{color,soul}

\usepackage{float}
\restylefloat{table}

\numberwithin{equation}{section}

\def\cocoa{{\hbox{\rm C\kern-.13em o\kern-.07em C\kern-.13em o\kern-.15em A}}}

\newtheorem{theorem}{Theorem}[section]

\newtheorem{lemma}[theorem]{Lemma}
\newtheorem{proposition}[theorem]{Proposition}

\theoremstyle{definition}
\newtheorem{remark}[theorem]{Remark}
\newtheorem{definition}[theorem]{Definition}
\newtheorem{example}[theorem]{Example}

\newtheorem{construction}[theorem]{Construction}

\newcommand {\sHom}{\mathcal{H}\kern -0.25ex{\mathit om}}
\newcommand {\sExt}{\mathcal{E}\kern -0.25ex{\mathit xt}}
\newcommand {\sTor}{\mathcal{T}\kern -0.25ex{\mathit or}}
\newcommand {\im}{\mathrm{im}}
\newcommand {\rk}{\mathrm{rk}}

\newcommand {\Ext}{\mathrm{Ext}}
\newcommand {\Hom}{\mathrm{Hom}}
\newcommand {\Hilb}{\mathcal{H}\kern -0.25ex{\mathit ilb\/}}

\newcommand {\quantum}{k}

\newcommand {\cK}{\mathcal{K}}
\newcommand {\cA}{\mathcal{A}}
\newcommand {\cB}{\mathcal{B}}

\newcommand{\cC}{{\mathcal C}}
\newcommand{\cS}{{\mathcal S}}
\newcommand{\cE}{{\mathcal E}}
\newcommand{\cF}{{\mathcal F}}
\newcommand{\cM}{{\mathcal M}}
\newcommand{\cN}{{\mathcal N}}
\newcommand{\cO}{{\mathcal O}}
\newcommand{\cG}{{\mathcal G}}

\newcommand{\cI}{{\mathcal I}}

\newcommand {\bZ}{\mathbb{Z}}

\newcommand {\bC}{\mathbb{C}}
\newcommand {\bP}{\mathbb{P}}

\newcommand {\bF}{\mathbb{F}}

\newcommand{\Pic}{\operatorname{Pic}}

\def\p#1{{\bP^{#1}}}

\def\mapright#1{\mathbin{\smash{\mathop{\longrightarrow}
\limits^{#1}}}}

\title[Even and odd instanton bundles]{Even and odd instanton bundles\\ 
on Fano threefolds}

\thanks{The first and second authors are a member of GNSAGA group of INdAM and are supported by the framework of the MIUR grant Dipartimenti di Eccellenza 2018-2022 (E11G18000350001). The third author is supported by the grant MAESTRO NCN - UMO-2019/34/A/ST1/00263 - Research in Commutative Algebra and Representation Theory.}



\author[V. Antonelli, G. Casnati, O. Genc]{Vincenzo Antonelli, Gianfranco Casnati, Ozhan Genc}

\begin{document}

\begin{abstract}
We define non--ordinary instanton bundles on Fano threefolds $X$ extending the notion of (ordinary) instanton bundles introduced in \cite{C--C--G--M}. We determine a lower bound for the quantum number of a non--ordinary instanton bundle, i.e. the degree of its second Chern class, showing the existence of such bundles for each admissible value of the quantum number when $i_X\ge 2$ or $i_X=1$, $\rk \Pic(X)=1$ and $X$ is ordinary. In these cases we deal with the component inside the moduli spaces of simple bundles containing the vector bundles we construct and we study their restriction to lines. Finally we give a monadic description of non--ordinary instanton bundles on $\p3$ and the smooth quadric studying their loci of jumping lines, when of the expected codimension.
\end{abstract}
\maketitle
\section{Introduction}
Let $X$ be a threefold over the complex field $\bC$. We say that $X$ is a {\sl Fano threefold} if its anticanonical line bundle $\omega_X^{-1}$ is ample. The greatest integer $i_X$ such that $\omega_X\cong\cO_X(-i_Xh)$ for some ample $\cO_X(h) \in \Pic(X)$ is called {\sl the index} of $X$: it is well--known that such an $\cO_X(h)$ is uniquely determined and it is called the {\sl fundamental line bundle} of $X$. 

The very first examples of Fano threefold are $\p3$ and the smooth quadric $Q\subseteq\p4$. In this cases $i_{\p3}=4$ and $i_Q=3$ respectively and the fundamental line bundle is the one cut out by the hyperplanes. Thus the origin of the study of vector bundles supported on Fano threefolds may be traced back to the seminal papers \cite{Ba,Ha4,El--St}, inspecting the case of the projective $3$--space, and to \cite{S--W,O--S} for $Q$. 

In the aforementioned papers, the authors focus their attention on $\mu$--stable bundles. Recall that the number $\mu(\cF)=c_1(\cF)h^{2}/\rk(\cF)$ is defined for each torsion--free sheaf $\cF$. The torsion--free sheaf $\cF$ is called {\sl $\mu$--semistable} (resp. {\sl $\mu$--stable}) if for all proper subsheaves $\cG$ with $0<\rk(\cG)<\rk(\cF)$ we have $\mu(\cG) \le \mu(\cF)$ (resp. $\mu(\cG)<\mu(\cF)$).

Among the $\mu$--stable rank two vector bundles on $\p3$, a relevant role is played by {\sl instanton bundles}, i.e. $\mu$--stable rank two vector bundles $\cE$ such that $c_1(\cE)=0$ and $h^1\big(\cE(-2)\big)=0$. Instanton bundles on $\p3$ carry many interesting properties and they have been thoroughly studied in the past and recent years.

Starting from $\p3$, the notion of instanton bundle has been widely generalized, e.g. to projective space and smooth quadrics of any dimension. A first generalization to the  Fano threefolds with Picard number $\varrho_X:=\rk \Pic (X)=1$ can be found in \cite{Fa} (see also \cite{Kuz}). More precisely, an instanton bundle is defined in \cite{Fa} as a $\mu$--stable rank $2$ bundle with $c_1(\cE)=(2q_X-i_X)h$ and $h^1\big(\cE(-q_X h)\big)=0$ where $q_X=\left[ \frac{i_X}{2}\right]$. The existence of instanton bundles according to this definition is proved for $i_X=2,3$ in \cite{Fa} and for $i_X=1$ in \cite{B--F2}.

In \cite{M--M--PL}, the authors extended the notion of instanton bundle to the flag manifold $F(0,1,2)$ which is a Fano threefold with $\varrho_X=2$. More recently, the definition has been further generalized to each Fano threefold, regardless of $i_X$ or $\varrho_X$: see \cite[Definition 1.2]{C--C--G--M}. Moreover, the existence of instanton bundles according to the latter definition has been settled  on some Fano threefolds: see \cite{C--C--G--M, Cs--Ge, A--M, A--C--G}. 

In any case, each instanton bundle $\cE$ on $X$ satisfies $c_1(\cE)\in\{0,-h\}$ according to the parity of $i_X$. E.g. $c_1(\cE)=0$ when $X\cong\p3$. Nonetheless, rank $2$ vector bundles $\cE$ with $c_1(\cE)=-1$ on $\p3$ are certainly interesting and they have been classical object of study (e.g. see \cite{H--S,B--M,El--Gr}) for low values of $c_2(\cE)$. 

Motivated by the above considerations, in this paper we give the following definition, where for $\varepsilon\in \left\{0,1\right\}$ we set 
$$
q_X^\varepsilon=\left[ \frac{i_X+1-\varepsilon}{2}\right]
$$
and we denote the integral cohomology of $X$ by $H^i(X)$.

\begin{definition}\label{dTwistedInstanton}
Let $X$ be a Fano threefold.

A vector bundle $\cE$ of rank 2 on $X$ is called an instanton bundle if the following properties hold:
\begin{itemize}
\item $c_1(\cE)=-\varepsilon h$, where $\varepsilon\in \left\{0,1\right\}$;
\item $\cE$ is $\mu$--semistable with respect to $\cO_X(h)$ and $(1-\varepsilon)h^0\big(\cE\big)=0$;
\item $h^1\big(\cE(-q_X^\varepsilon h)\big)=0$.
\end{itemize}
The class $c_2(\cE)\in H^4(X)$ and its degree are respectively called (topological) charge and quantum number of $\cE$.
\end{definition}

We spend now a few words in order to better explain Definition \ref{dTwistedInstanton} above. When $\varepsilon=1$ we do not require the vanishing $h^0\big(\cE\big)=0$ in the above definition. On the one hand, in spite of this apparent asymmetry, the vanishing of $h^0\big(\cE\big)$ when $\varepsilon=1$ is actually a direct consequence of the $\mu$--semistability of $\cE$, because $\mu(\cE)<0$. On the other hand, when $\varepsilon =0$ the vector bundle $\cE:=\cO_X^{\oplus2}$ certainly satisfies all the properties of Definition \ref{dTwistedInstanton} but $h^0\big(\cE\big)=0$. Moreover, when $i_X\le2$ it is not difficult to give further examples of non--trivial rank $2$ bundles with an analogous behaviour and with arbitrarily large quantum number: e.g. every rank $2$ bundle associated via the Serre correspondence to a disjoint set of integral curves of degree $3-i_X$ in $X$.

In order to have a rough idea of the meaning of the {\sl instantonic condition}, i.e. the vanishing condition for the degree $1$ cohomology in Definition \ref{dTwistedInstanton} above, we will show in Lemma \ref{lH12} that $h^1\big(\cE(th)\big)=0$ for $t\le-q_X^\varepsilon$ for each $\mu$--semistable rank $2$ vector bundle $\cE$ on $X$ with $c_1(\cE)=-\varepsilon h$, $\varepsilon\in \left\{0,1\right\}$. If $X\cong\p3$, thanks to \cite[Theorem 7]{Mad} the vanishing
$$
h^1\big(\cE((1-q_{\p3}^\varepsilon)h)\big)=h^1\big(\cE(-h)\big)=0,
$$
implies that $\cE$ is a direct sum of line bundles: the same is true for $X\cong Q$  and $\varepsilon=0$.  In particular, in these cases, $-q_X^\varepsilon=-2$ is the greatest integer such that there could exist an indecomposable $\cE$ as above with $h^1\big(\cE(th)\big)=0$ for $t\le-q_X^\varepsilon$.

\medbreak
In Section \ref{sGeneral} we list some helpful facts about vector bundles and Fano threefolds that we will use throughout the whole paper. 

\medbreak
When $\varrho_X=1$, each instanton bundle $\cE$ is $\mu$--stable, hence {\sl simple}, i.e.
$$
\Hom_X\big(\cE,\cE\big)\cong H^0\big(\cE\otimes\cE^\vee\big)\cong\bC.
$$
When $\varrho_X\ge2$ this is no longer true. Nevertheless in Section \ref{sSimple} we prove the following result characterizing the group of endomorphisms of an instanton bundle.

\begin{theorem}
\label{tSimple}
Let $X$ be a Fano threefold.

If $\cE$ is an instanton bundle on $X$, the following assertions hold.
\begin{enumerate}
\item If $\cE$ is  indecomposable, then it is simple.
\item If $\cE$ is decomposable, then $\cE\cong\cO_{X}(D)\oplus\cO_{X}(-D-\varepsilon h)$,
where $Dh^2=-\varepsilon h^3/2$ and
\begin{equation}
\label{VanishingDecomposable}
\begin{gathered}
 h^0\big(\cO_{X}(D)\big)=h^0\big(\cO_{X}(-D-\varepsilon h)\big)=0,\\
 h^1\big(\cO_{X}(D-q_X^\varepsilon h)\big)=h^1\big(\cO_{X}(-D-(\varepsilon+q_X^\varepsilon)h)\big)=0.
\end{gathered}
\end{equation}
In this case, $\Hom_X\big(\cE,\cE\big)\cong\bC^{\oplus2}$ acting diagonally.
\end{enumerate}
\end{theorem}

\medbreak
When $i_X-\varepsilon$ is even, instanton bundles in the sense of Definition \ref{dTwistedInstanton} coincide with instanton bundles as defined in \cite{Fa} and \cite{C--C--G--M}, because $q_X=q_X^\varepsilon$. For this reason, we give the following definition.

\begin{definition}
Let $X$ be a Fano threefold.

An instanton bundle $\cE$ on $X$ with $c_1(\cE)=-\varepsilon h$ is called even or odd (resp. ordinary, non--ordinary), if $\varepsilon$ (resp. $i_X-\varepsilon$) is even or odd respectively.
\end{definition}

As we already pointed out above, ordinary instanton bundles have been widely studied. The present paper is essentially focused on the non--ordinary instanton bundle, studying their elementary properties, constructing examples and studying the corresponding points in the moduli spaces where they sit.

\medbreak
More precisely, in Section \ref{sSharp} we deal with non--ordinary instanton bundles and we prove the following lower bound for the quantum numbers of such bundles (see \cite[Lemma 4.2]{C--C--G--M} for a similar bound in the ordinary case).

\begin{theorem}
\label{tSharp}
Let $X$ be a Fano threefold.

If $\cE$ is an indecomposable non--ordinary instanton bundle on $X$, then its quantum number is even and satisfies
\begin{equation}
\label{BoundSharp}
\quantum\ge\left\lbrace\begin{array}{ll} 
2\quad&\text{if either $i_X\ge3$ or $i_X=2$ and $h^3\le5$,}\\
4\quad&\text{if either $i_X=2$ and $h^3\ge6$ or $i_X=1$.}
\end{array}\right.
\end{equation}
\end{theorem}

\medbreak
The above lower bound is actually sharp when $i_X\ge2$. Indeed, in Construction \ref{conInstanton} we use the Serre correspondence to associate certain bundles of rank $2$ to a suitable set of pairwise disjoint conics, which enables us to prove the following result in Section  \ref{sExistence} extending \cite[Example 3.1.2]{Ha4}.

\begin{theorem}
\label{tExistence}
Let $X$ be a Fano threefold with $i_X\ge2$ and very ample $\cO_X(h)$.

For each even integer $\quantum$ satisfying Inequality \eqref{BoundSharp}, Construction \ref{conInstanton} gives a $\mu$--stable, non--ordinary instanton bundle $\cE$ with quantum number $\quantum$. 
\end{theorem}

\medbreak
When $i_X=1$ the existence of even instanton bundles is slightly less immediate. Indeed Construction \ref{conInstanton} returns rank $2$ bundles which are certainly not instanton, because they have non--zero sections. In particular the approach via Serre correspondence is certainly not possible if $X$ is a {\sl prime Fano threefold}, i.e. a Fano threefold such that $\varrho_X=1$ as we show at the beginning of Section \ref{sPrime}. 

We also describe therein a completely different approach based on a suitable chain of deformations of the bundles obtained via Construction \ref{conInstanton}, showing that such an approach leads to instanton bundles and proving the following result extending \cite[Theorem 4.1]{B--F1} at least when $X$ is {\sl ordinary}, i.e. it contains a line with normal bundle $\cO_{\p1}\oplus\cO_{\p1}(-1)$.

\begin{theorem}
\label{tPrime}
Let $X$ be an ordinary prime Fano threefold  with very ample $\cO_X(h)$.

For each even integer $\quantum\ge4$, there exists a $\mu$--stable, even instanton bundle with quantum number $\quantum$. 
\end{theorem}

Again $\quantum=4$ is the minimal admissible value for the quantum number of an even instanton bundle on an ordinary prime Fano threefold given by Theorem \ref{tSharp}. Thus Theorem \ref{tSharp} is sharp also if $i_X=1$, at least in the case $\varrho_X=1$. 

The problem of the existence of instanton bundles when $\varrho_X\ge2$ remains wide open in general, though it is not difficult to show the existence of even instanton bundles in several specific cases via Serre construction. Finally, in Remark \ref{rCanonical} we  suggest a completely different possible approach for constructing minimal instanton bundles when $\varrho_X\ge2$.
 
\medbreak
An important property of instanton bundles on $\p3$, occasionally assumed in some classical definitions, is that they are trivial when restricted to the general line, thanks to the Grauert--M\"ulich theorem (see \cite[Corollary 2 of Theorem II.2.1.4]{O--S--S}). A line on $X$ is a curve $L$ with Hilbert polynomial $t+1$, i.e. such that $\deg(L)=1$ and $p_a(L)=0$: if $\cO_X(h)$ is very ample, then lines on $X$ are exactly the curves which are lines with respect to the embedding induced by $\cO_X(h)$. In what follows we will denote by $\Lambda(X)$ the Hilbert scheme of lines on $X$.

\begin{definition}
Let $X$ be a Fano threefold.

We say that an instanton bundle $\cE$ on $X$ with $c_1(\cE)=-\varepsilon h$ is generically trivial on the component $\Lambda(X)_0 \subseteq \Lambda(X)$, if $h^1\big(\cE((i_X+\varepsilon-2q_X^\varepsilon -1)h)\otimes \cO_L\big)=0$ when $L\in\Lambda(X)_0$ is general. $\cE$ is generically trivial, if it is generically trivial on each component.
\end{definition}

When $i_X\ge2$, then each Fano threefold is covered by lines, while this is no longer true if $i_X=1$. 

Generically trivial ordinary instanton bundles have been constructed on several Fano threefolds (\cite{Fa,M--M--PL,A--M,Cs--Ge, A--C--G}). Thus it is quite natural to analyze the behaviour of the instanton bundles whose existence is guaranteed by Theorems \ref{tExistence} and \ref{tPrime} when restricted to general lines: we make such an analysis in Section \ref{sEarnest} where we prove the following result.
 
\begin{theorem}
\label{tTrivial}
Let $X$ be a Fano threefold  with very ample $\cO_X(h)$.

Then the bundles obtained via Construction \ref{conInstanton} are generically trivial. Moreover, the general bundle as in Theorem \ref{tPrime} is generically trivial. 
\end{theorem}

In the same section, we also show that such bundles also behave well when restricted to smooth and irreducible divisors (see Proposition \ref{pEarnest}).
 
\medbreak
Theorem \ref{tSimple} implies that indecomposable instanton bundles correspond to points inside the moduli space $\cS_X$ of simple torsion--free sheaves (see \cite{A--K} for details about such a space). Moreover, the instantonic condition implies by semicontinuity that such points fill an open subset $\cS\cI_X$.

When $X\cong\p3$, the moduli space of instanton bundles with fixed charge is irreducible (see \cite{Tik1,Tik2}) and smooth (see \cite{J--V}). In \cite{Fa}, the author proves that for many other families of Fano threefolds $X$ with $\varrho_X=1$ the moduli space of instanton bundles with fixed charge has a generically smooth irreducible component. Similar results have been obtained in \cite{M--M--PL,C--C--G--M,A--M,Cs--Ge, A--C--G} for other families of Fano threefolds. 

In Section \ref{sModuli} we check that the instanton bundles given by Construction \ref{conInstanton} represent smooth points of one and the same component of their moduli space (see Proposition \ref{pModuli}). 

An instanton bundle is called {\sl minimal}, if its quantum number $\quantum$ is as small as possible: they typically carry additional properties (e.g. see \cite[Remark 4.3]{C--C--G--M}). As a by--product of Theorem \ref{tExistence} we infer that the moduli space of minimal non--ordinary instanton bundles $\cE$ on Fano threefolds $X$ with $i_X=2$ and $\varrho_X=1$ is irreducible and smooth, because such bundles are the ones studied in \cite{S--W} (see Remark \ref{rMinimalIndex1}). In particular such minimal non--ordinary instanton bundles are  {\sl aCM (with respect to $\cO_X(h)$)}, i.e.  $h^i\big(\cE(th)\big)=0$ for $i=1,2$ and each $t\in\bZ$ (see also \cite{A--C}). 

Moreover, we also show that each instanton bundle obtained as in Theorem \ref{tPrime} is a smooth point inside the component of the moduli space which contains it (see Proposition \ref{pModuliPrime}).

\medbreak
Finally, in the last Section \ref{sMonad}, we focus on non-ordinary instanton bundles on the two Fano threefolds with $i_X\ge 3$. In particular we show that non-ordinary instanton bundles on $\p3$ can be realized as cohomology of a monad, similarly to the classical ordinary case. Using such a monad we are also able to explicitly describe the locus of jumping lines for non--ordinary instanton bundles.

\subsection{Acknowledgements}
The authors express their thanks the referee for her/his comments, remarks and suggestions which have considerably improved the whole exposition correcting several inaccuracies. The authors also thank R. Notari for a fruitful discussion about the proof of Proposition \ref{pStable}. 

\section{General results}
\label{sGeneral}
We list below some general helpful results used throughout the whole paper. Let $X$ be any smooth variety with canonical line bundle $\omega_X$.

Let $\cF$ be a vector bundle of rank $2$ on $X$ and let $s\in H^0\big(\cF\big)$. In general its zero--locus
$(s)_0\subseteq X$ is either empty or has codimension at most
$2$. We can always write $(s)_0=Z\cup S$
where $Z$ has codimension $2$ (or it is empty) and $S$ has pure codimension
$1$ (or it is empty). In particular $\cF(-S)$ has a section vanishing
on $Z$, thus we can consider its Koszul complex 
\begin{equation}
  \label{seqSerre}
  0\longrightarrow \cO_X(S)\longrightarrow \cF\longrightarrow \cI_{Z\vert X}(-S)\otimes\det(\cF)\longrightarrow 0.
\end{equation}
Moreover we also have the standard exact sequence
\begin{equation}
\label{seqStandard}
 0 \longrightarrow\cI_{Z\vert X} \longrightarrow \cO_X \longrightarrow \cO_Z \longrightarrow 0
\end{equation}
Sequence \ref{seqSerre} tensored by $\cO_Z$ yields $\cI_{Z\vert X}/\cI^2_{Z\vert X}\cong\cF^\vee(S)\otimes\cO_Z$, whence the normal bundle of $Z$ inside $X$ satisfies
\begin{equation}
\label{Normal}
\cN_{Z\vert X}\cong\cF(-S)\otimes\cO_Z.
\end{equation}
Thus $Z$ is locally complete intersection inside $X$, because $\rk(\cF)=2$. In particular, it has no embedded components.

The Serre correspondence allows us to revert the above construction as follows.

\begin{theorem}
  \label{tSerre}
  Let $Z\subseteq X$ be a local complete intersection subscheme of codimension $2$.
  
  If $\det(\cN_{Z\vert X})\cong\cO_Z\otimes\mathcal L$ for some $\mathcal L\in\Pic(X)$ such that $h^2\big(\mathcal L^\vee\big)=0$, then there exists a vector bundle $\cF$ of rank $2$ on $X$ such that:
  \begin{enumerate}
  \item $\det(\cF)\cong\mathcal L$;
  \item $\cF$ has a section $s$ such that $Z$ coincides with the zero--locus $(s)_0$ of $s$.
  \end{enumerate}
  Moreover, if $H^1\big({\mathcal L}^\vee\big)= 0$, the above two conditions  determine $\cF$ up to isomorphism.
\end{theorem}
\begin{proof}
See \cite{Ar}.
\end{proof}

If $\cF$ and $\cG$ are coherent sheaves on $X$, then  the Serre duality holds
\begin{equation}
\label{SerreDual}
\Ext_X^i\big(\cF,\cG\otimes\omega_X\big)\cong \Ext_X^{\dim(X)-i}\big(\cG,\cF\big)^\vee
\end{equation}
(see \cite[Proposition 7.4]{Ha3}). Moreover, we define
$$
\chi(\cG,\cF):=\sum_{i=0}^{\dim(X)}\dim\Ext^i_X\big(\cG,\cF\big).
$$
Notice that if $\cG$ is locally free, then $\chi(\cG,\cF)=\chi(\cF\otimes\cG^\vee)$. If $\cF$ is locally free, Equality \eqref{SerreDual} gives $\chi(\cG,\cF)=(-1)^{\dim(X)}\chi(\cG\otimes\cF^\vee\otimes\omega_X)$.

Consider a locally complete intersection subscheme $Z\subseteq X$. Then
\begin{equation}
\label{Local}
\sExt^i_X\big(\cO_Z,\cO_Z\big)\cong\sExt^{i-1}\big(\cI_{Z\vert X},\cO_Z\big)\cong\wedge^i\cN_{Z\vert X}
\end{equation}
where $\cN_{Z\vert X}^\vee:=\cI_{Z\vert X}/\cI_{Z\vert X}^2\cong\cI_{Z\vert X}\otimes\cO_Z$.

The Riemann--Roch formula for a vector bundle $\cF$ on a threefold $X$ is
\begin{equation}
 \label{RRGeneral}
\begin{aligned}
     \chi(\cF)&=\rk(\cF)\chi(\cO_X)+{\frac16}(c_1(\cF)^3-3c_1(\cF)c_2(\cF)+3c_3(\cF))\\
    &-{\frac14}(\omega_Xc_1(\cF)^2-2\omega_Xc_2(\cF))+{\frac1{12}}(\omega_X^2c_1(\cF)+c_2(\Omega_{X}) c_1(\cF))
\end{aligned}
\end{equation}
(see \cite[Theorem A.4.1]{Ha2}).

We now assume that $X$ is a Fano threefold. 
The following lemma will be helpful.

\begin{lemma}
\label{lExt3}
Let $X$ be a Fano threefold.

If $\cF$ is a simple vector bundle on $X$, then $\Ext^{3}\big(\cF,\cF\big)=0$.
\end{lemma}
\begin{proof}
Equality \eqref{SerreDual} implies $\Ext^{3}_{X}\big(\cF,\cF\big)^\vee\cong\Hom_{X}\big(\cF,\cF(-i_Xh)\big)$. In order to prove the statement, it then suffices to check that the latter space is zero. On the one hand, $\varphi\in\Hom_{X}\big(\cF,\cF(-i_Xh)\big)$, hence $\det(\varphi)\in H^0\big(\cO_X(-i_X\rk(\cF)h)\big)=0$. On the other hand $\Hom_{X}\big(\cF,\cF(-i_Xh)\big)\subseteq\Hom_{X}\big(\cF,\cF\big)$, because $h^0\big(\cO_X(h)\big)\ne0$ (see \cite[Theorem 2.3.1]{I--P}). As $\cF$ is simple, each non zero endomorphism of $\cF$ is an automorphism. We deduce  that necessarily $\varphi=0$.
\end{proof}

We are interested in $\mu$--semistable bundles, hence the following lemmas will be useful.

\begin{lemma}
\label{lHoppe}
Let $X$ be  a Fano threefold.

If $\cF$ is a rank $2$ vector bundle on $X$, then $\cF$ is $\mu$--stable (resp. $\mu$--semistable) with respect to $\cO_X(h)$ if and only if  $h^0\big(\cF(-D)\big)=0$ for each divisor $D\subseteq X$ such that $Dh^2 \ge\mu(\cF)$ (resp. $>\mu(\cF)$).
\end{lemma}
\begin{proof}
The group $\Pic(X)$ is free, hence it suffices to apply \cite[Corollary 4]{J--M--P--S}.
\end{proof}

The following proposition partially clarifies the importance of the number $q_X^\varepsilon$ defined in the introduction.

\begin{lemma}
\label{lH12}
Let $X$ be a Fano threefold with very ample $\cO_X(h)$.

If $\cE$ is a rank $2$ $\mu$--semistable bundle on $X$ such that $c_1(\cE)=-\varepsilon h$ with $\varepsilon\in \left\{0,1\right\}$ and $h^1\big(\cE(-q_X^\varepsilon h)\big)=0$, then 
\begin{enumerate}
\item $h^1\big(\cE(th)\big)=0$ for $t\le -q_X^\varepsilon$;
\item $h^2\big(\cE(th)\big)=0$ for $t\ge-q_X^\varepsilon+1$ when $i_X-\varepsilon$ is odd and for $t\ge-q_X^\varepsilon$ when $i_X-\varepsilon$ is even.
\end{enumerate}
\end{lemma}
\begin{proof}
If $D$ is a general surface of degree $-t-q_X^\varepsilon\ge1$, then $\cE\otimes\cO_D$ is $\mu$--semistable, thanks to \cite[Theorem 3.1]{Ma1}. If $i_X-\varepsilon\ne0$, then $q_X^\varepsilon\ge1$ and Lemma \ref{lHoppe} implies
\begin{equation}
\label{Restriction}
h^0\big(\cE(-q_X^\varepsilon h)\otimes\cO_D\big)=0.
\end{equation}
If $i_X-\varepsilon=0$, then $\varepsilon=1$ and $q_X^\varepsilon=0$. In this case $\mu(\cE\otimes\cO_D)=h^2D/2<0$, hence again Equality \eqref{Restriction} above holds true. The cohomology of the exact sequence
\begin{equation}
\label{seqMaruyama}
0\longrightarrow\cO_X(-D)\longrightarrow\cO_X\longrightarrow\cO_D\longrightarrow0
\end{equation}
tensored by $\cE(-q_X^\varepsilon h)$  yields the first assertion.

Thanks to Equality \eqref{SerreDual} we also obtain $h^2\big(\cE(th)\big)=h^1\big(\cE(-th)\big)=0$ for $t\ge-q_X^\varepsilon+1$ when $i_X-\varepsilon$ is odd and for $t\ge-q_X^\varepsilon$ when $i_X-\varepsilon$ is even.
\end{proof}

A helpful characterization of ordinary instanton bundles on a Fano threefold is related to the naturality of their cohomology in a certain range. We show below that such a characterization actually holds also in the non--ordinary case.

\begin{proposition}
\label{pNatural}
Let $X$ be a Fano threefold.

Every $\mu$--semistable bundle $\cE$ of rank $2$ on $X$ such that $(1-\varepsilon)h^0\big(\cE\big)=0$,  $c_1(\cE)=-\varepsilon h$ and with natural cohomology in degree $-q_X^\varepsilon$ is an instanton bundle.

Every instanton bundle $\cE$ on $X$ has natural cohomology in degree $t$ in the range $\varepsilon-i_X\le t \le 0$. More precisely, $h^i\big(\cE(\lambda h)\big)=0$ unless either $1-q_X^\varepsilon\le t\le0$ and $i=1$, or $\varepsilon-i_X\le t\le\varepsilon-i_X-1+q_X^\varepsilon$ and $i=2$.
\end{proposition}
\begin{proof}
The case when $i_X-\varepsilon$ is even is covered by \cite[Proposition 4.1]{C--C--G--M}, hence we can restrict our attention to the odd $i_X-\varepsilon$ case.
In this case $2q_X^\varepsilon=i_X+1-\varepsilon\ge2$, hence $q_X^\varepsilon\ge1$. If $\varepsilon=1$, then we know that $h^0\big(\cE\big)=0$, because $\cE$ is $\mu$--semistable. If $\varepsilon=0$ the same vanishing holds by definition.

Assume that $\cE$ has natural cohomology in degree $-q_X^\varepsilon$. We have
\begin{gather*}
h^0\big(\cE(-q_X^\varepsilon h)\big)\le h^0\big(\cE\big)=0,\qquad h^3\big(\cE(-q_X^\varepsilon h)\big)=h^0\big(\cE((1-q_X^\varepsilon) h)\big)\le h^0\big(\cE\big)=0.
\end{gather*}
If $\cO_X(dh)$ is very ample, then $\cE\otimes\cO_S$ is $\mu$--semistable for general $S\in \vert dh\vert$, thanks to \cite[Theorem 3.1]{Ma1}. Thus $4hc_2(\cE)\ge \varepsilon^2h^3$ by Bogomolov inequality. A case by case computation yields $\chi(\cE(-q_X^\varepsilon h))\ge0$. We deduce that $h^1\big(\cE(-q_X^\varepsilon h)\big)=0$ necessarily, hence $\cE$ is an instanton bundle.

Conversely, let $\cE$ be an instanton bundle. Equality \eqref{SerreDual} implies
\begin{equation}
\label{Natural}
h^i\big(\cE(t h)\big)=h^{3-i}\big(\cE(-(i_X-\varepsilon +t)h)\big)=h^{3-i}\big(\cE((1-2q_X^\varepsilon -t)h)\big).
\end{equation}
Thus the cohomology of $\cE$ in the range $\varepsilon-i_X=1-2q_X^\varepsilon\le t\le 0$ is known when it is known in the range  $1-q_X^\varepsilon\le t\le 0$. 

If $i_X-\varepsilon=1$, then $q_X^\varepsilon=1$, hence $t=0$. The assertion then follows from the hypothesis. If $i_X-\varepsilon=3$, then $q_X^\varepsilon=2$, hence $t=0,1$. Moreover, in this case $X$ is either $\p3$ or $Q$, hence $\cO_X(h)$ is very ample. Thus the statement follows by combining the definition, Equality \eqref{Natural} and Lemma \ref{lH12}.\end{proof}

\section{Endomorphism group of instanton bundles}
\label{sSimple}
In this section we deal with the endomorphism group of an instanton bundle. The lemma below is well--known.

\begin{lemma}
\label{lSimple}
Let $X$ be a smooth variety of dimension $n\ge2$ endowed with an ample line bundle $\cO_X(h)$. 

If $\cE$ is a strictly $\mu$--semistable bundle of rank $2$ on $X$, then it fits into an exact sequence
\begin{equation}
\label{seqExtensionCurve}
 0 \longrightarrow\cO_X(D) \longrightarrow \cE \longrightarrow \cI_{Z\vert X}(-D)\otimes\det(\cE) \longrightarrow 0
\end{equation}
where $\mu(\cO_X(D))=\mu(\cE)$ and $Z$ is either empty or of pure codimension $2$.
\end{lemma}
\begin{proof}
By definition, there is a subsheaf $\mathcal L\subseteq\cE$ of rank $1$ such that $\mu(\mathcal L)=\mu(\cE)$. 

Let $\mathcal Q:=\cE/\mathcal L$ and denote by $\cM$ the kernel of the natural morphism $\cE\to\mathcal Q/\mathcal T$, where $\mathcal T\subseteq\mathcal Q$ is the torsion subsheaf. By construction $\cM$ is torsion--free and of rank $1$: moreover, $\mu(\cM)\ge\mu(\mathcal L)=\mu(\cE)$, hence $\mu(\cM)=\mu(\cE)$ because $\cE$ is $\mu$--semistable. Since $\cE/\cM$ is torsion--free, it follows that $\cM$ is also normal thanks to \cite[Lemma II.1.1.16]{O--S--S}. Thus \cite[Lemma II.1.1.12]{O--S--S} implies that $\cM$ is reflexive, hence it is a line bundle, thanks to \cite[Lemma II.1.1.15]{O--S--S}. We deduce that $\cM\cong\cO_X(D)$ for some divisor $D$ on $X$.

Let $s\in H^0\big(\cE(-D)\big)$ be a non--zero section: with the notation introduced in the previous section we write  $(s)_0=Z\cup S$ and we can construct Sequence \eqref{seqSerre} with $\cF:=\cE(-D)$. In particular we have an injective map  $\cO_S(D+S)\to\cE$ and we can write the chain of inequalities
$$
0\le Sh^{n-1}=(D+S)h^{n-1}-Dh^{n-1}\le\mu(\cE)-\mu(\cO_X(D))=0,
$$
i.e. $S=0$, because $\cO_X(h)$ is ample (we are using the Nakai--Moishezon criterion: see \cite[Theorem A.5.1]{Ha2} and the references therein). Thus tensoring Sequence \eqref{seqSerre} with $\cF:=\cE(-D)$  by $\cO_X(D)$ we obtain Sequence \eqref{seqExtensionCurve}.
\end{proof}

Theorem \ref{tSimple} stated in the Introduction is an immediate by--product of the following proposition.

\begin{proposition}
\label{pSimple}
Let $X$ be a smooth variety of dimension $n\ge2$ endowed with an ample line bundle $\cO_X(h)$. Assume that $\Pic(X)$ has no non--trivial $2$--torsion elements and $\cO_X(h)$ is not $2$--divisible.

If $\cE$ is a rank $2$ indecomposable $\mu$--semistable bundle on $X$ such that $c_1(\cE)=-\varepsilon h$ with $\varepsilon\in\{\ 0,1\ \}$ and $(1-\varepsilon)h^0\big(\cE\big)=0$, then $\cE$ is simple.
\end{proposition}
\begin{proof}
Trivially $h^0\big(\cE\otimes\cE^\vee\big)\ge1$, hence we have to prove the opposite inequality assuming $\cE$ indecomposable. If $\cE$ is $\mu$--stable, then it is simple thanks to \cite[Theorem II.1.2.9]{O--S--S}. Thus we will assume that $\cE$ is strictly $\mu$--semistable. 

Lemma \ref{lSimple} yields the existence of a divisor $D$ on $X$ with
\begin{equation}
\label{Mu}
Dh^{n-1}=\mu(\cM)=\mu(\cE)=\frac{-\varepsilon h^n}2.
\end{equation}
and of Sequence \eqref{seqExtensionCurve}.
 Tensoring it with $\cE^\vee\cong\cE(\varepsilon h)$ and taking its cohomology we obtain
\begin{equation}
\label{Simple}
h^0\big(\cE\otimes\cE^\vee\big)\le h^0\big(\cE(D+\varepsilon h)\big)+h^0\big(\cE\otimes\cI_{Z\vert X}(-D)\big).
\end{equation}

We first show that $h^0\big(\cE(D+\varepsilon h)\big)=0$. The cohomology of Sequence \eqref{seqExtensionCurve} tensored by $\cO_X(D+\varepsilon h)$ yields the exact sequence
\begin{align*}
0\longrightarrow H^0\big(\cO_X(2D+\varepsilon h)\big)&\longrightarrow H^0\big(\cE(D+\varepsilon h)\big)\\
&\longrightarrow H^0\big(\cI_{Z\vert X}\big)\mapright\partial H^1\big(\cO_X(2D+\varepsilon h)\big).
\end{align*}
Equality \eqref{Mu} implies $(2D+\varepsilon h)h^{n-1}=0$. 

If $h^0\big(\cO_X(2D+\varepsilon h)\big)\ne0$, then the Nakai--Moishezon criterion would return $2D+\varepsilon h=0$, because $\cO_X(h)$ is ample. If $\varepsilon =0$, hence $2D=2D+\varepsilon h=0$, i.e. $D=0$, because $\Pic(X)$ has no non--trivial $2$--torsion elements. The cohomology of Sequence \eqref{seqExtensionCurve} implies
$$
h^0\big(\cE\big)=(1-\varepsilon)h^0\big(\cE\big)\ge h^0\big(\cO_X(D)\big)=1,
$$
a contradiction. If $\varepsilon =1$, then  $h=-2D$, contradicting that $\cO_X(h)$ is not $2$--divisible. 
Thus
\begin{equation}
\label{VanPos}
h^0\big(\cO_X(2D+\varepsilon h)\big)=0,
\end{equation}
hence $H^0\big(\cE(D+\varepsilon h)\big)\subseteq\ker(\partial)\subseteq H^0\big(\cI_{Z\vert X}\big)$.

If $Z\ne\emptyset$, then  $h^0\big(\cI_{Z\vert X}\big)=0$. If $Z=\emptyset$, then $h^0\big(\cI_{Z\vert X}\big)=h^0\big(\cO_X\big)=1$. Since $\partial$ must be non--zero, because $\cE$ is indecomposable, it follows that $\ker(\partial)=0$. In both cases we finally infer the vanishing
\begin{equation}
\label{Van1}
h^0\big(\cE(D+\varepsilon h)\big)=0.
\end{equation}

We now prove $h^0\big(\cE\otimes\cI_{Z\vert X}(-D)\big)\le1$. The cohomology of Sequence \eqref{seqExtensionCurve} tensored by $\cO_X(-D)$ yields 
\begin{equation}
\label{Ideal}
\begin{aligned}
h^0\big(\cE\otimes\cI_{Z\vert X}(-D)\big)\le h^0\big(\cE(-D)\big)&\le h^0\big(\cO_X\big)+h^0\big(\cI_{Z\vert X}(-2D-\varepsilon h)\big)\\
&\le 1+h^0\big(\cO_X(-2D-\varepsilon h)\big).
\end{aligned}
\end{equation}
We have $(-2D-\varepsilon h)h^{n-1}=0$, hence as above we again deduce
\begin{equation}
\label{VanNeg}
h^0\big(\cO_X(-2D-\varepsilon h)\big)=0.
\end{equation}
By combining Inequalities \eqref{Simple} and \eqref{Ideal} with Equalities \eqref{Van1} and \eqref{VanNeg}, we finally obtain $h^0\big(\cE\otimes\cE^\vee\big)\le1$.
\end{proof}

The above proposition leads to the following proof of Theorem \ref{tSimple} stated in the Introduction.

\medbreak
\noindent{\it Proof of Theorem \ref{tSimple}.}
If $\cE$ is indecomposable, then the statement follows from Proposition \ref{pSimple}, because on a Fano threefold $X$ the fundamental line bundle $\cO_X(h)$ is not divisible by definition and $\Pic(X)$ is torsion--free (see \cite{I--P}).

Let $\cE$ be decomposable. Thus there is a divisor $D\subseteq X$ such that $\cE\cong\cO_{X}(D)\oplus\cO_{X}(-D-\varepsilon h)$: necessarily $Dh^2=-\varepsilon h^3/2$, otherwise $\cE$ would be unstable. Moreover, Equalities \eqref{VanishingDecomposable} follow from the vanishing $h^0\big(\cE\big)=0$ and the instantonic condition. 

Equalities \eqref{VanPos} and \eqref{VanNeg} imply that $\Hom_X\big(\cE,\cE\big)$ contains only scalar multiples of the identity.
\qed
\medbreak

Notice that  the existence of decomposable instanton bundles strongly depends on the choice of the Fano threefold $X$, as the following example shows.

\begin{example}
Decomposable instanton bundles exist on the flag threefold (see \cite[Proposition 2.5]{M--M--PL}): it is easy to check that that the same is true on $\p1\times\p1\times\p1$ (see \cite{A--M}). On the other hand, the blow up of $\p3$ at a point does not support any decomposable instanton bundle (see \cite[Proposition 6.5]{C--C--G--M}). 
\end{example}

If $\cE$ is an indecomposable instanton bundle on a Fano threefold $X$ we then have  $h^0\big(\cE\otimes\cE^\vee\big)=1$. Moreover, $h^3\big(\cE\otimes\cE^\vee\big)=\dim\Ext^3_{X}\big(\cE,\cE\big)=0$ thanks to Lemma \ref{lExt3}. Thus Equality \eqref{RRGeneral} for $\cE\otimes\cE^\vee$ yields
\begin{equation}
\label{Ext12}
h^1\big(\cE\otimes\cE^\vee\big)-h^2\big(\cE\otimes\cE^\vee\big)=2i_Xc_2(\cE) h-\frac{i_X\varepsilon\deg(X)}2-3.
\end{equation}

\section{Quantum numbers of non--ordinary instanton bundles}
\label{sSharp}
Before proving Theorem \ref{tSharp} we first briefly recall the classification of  Fano threefolds with $i_X\ge2$ and very ample $\cO_X(h)$ fixing the notation (see \cite{I--P} and the references therein).
\begin{itemize}
\item $X\cong\p3$. In this case $h^3=1$, $i_X=4$ and $\varrho_X=1$.
\item $X$ is a smooth quadric $Q\subseteq\p4$. In this case $h^3=2$, $i_X=3$ and  $\varrho_X=1$. 
\item $X$ is a smooth cubic in $F_3\subseteq\p4$. In this case $h^3=3$, $i_X=2$ and  $\varrho_X=1$. 
\item $X$ is the smooth complete intersection of two quadrics  in $F_4\subseteq\p5$. In this case $h^3=4$, $i_X=2$ and  $\varrho_X=1$.
\item $X$ is the smooth and proper intersection of the grassmannian  of lines in $\p4$, $F_5:=G(2,5)\subseteq\p9$, with a linear subspace in $\p9$ of dimension $3$. In this case $h^3=5$, $i_X=2$ and  $\varrho_X=1$.
\item $X$ is any smooth intersection $F_{6,2}$ of the image of the Segre embedding of $\p2\times\p2$ inside $\p8$ with a hyperplane $H$. In this case $h^3=6$, $i_X=2$ and $\varrho_X=2$. 
\item $X$ is the image $F_{6,3}$ of the Segre embedding of $\p1\times\p1\times\p1$ inside $\p7$. In this case $h^3=6$, $i_X=2$ and $\varrho_X=3$. 
\item $X$ is the blow up $F_7$ of $\p3$ at a point $p$ embedded in $\p8$ via the linear system induced by the quadrics through $p$. In this case $h^3=7$, $i_X=2$ and $\varrho_X=2$. 
\end{itemize}

Recall that we already introduced in the Introduction the notion of line on a Fano threefold and we denoted by $\Lambda(X)$ their Hilbert scheme. A conic $C\subseteq X$ is a curve with Hilbert polynomial $2t+1$ or, in other words, such that $\deg(C)=2$ and $p_a(C)=0$. From now on $\Gamma(X)$ will denote the Hilbert scheme of conics in $X$.

It is helpful to identify lines and conics on $X$. We exploit such an identification in the following remarks that will be used several times in the paper. In what follows $ H^*(X)$ denotes the integral cohomology ring of $X$.

\begin{remark}
\label{rR=1}
Let $\varrho_X=1$. 

If $X\cong\p3$, then $H^*(X)\cong\bZ[h]/(h^4)$, where $h$ is the class of a plane. The classes of lines and conics on $X$ are $\ell:=h^2$ and $2h^2$ respectively. By definition $\Lambda(X)$ is the Grassmannian $G(2,4)$ of lines in $\p3$. The scheme $\Gamma(X)$ is a projective bundle on $\vert h\vert$ with fibre isomorphic to the $\p5$ of conics in a plane.

If $X$ is the smooth quadric $Q\subseteq\p4$, then $H^*(X)\cong\bZ[h]/(h^2-2\ell,h^4)$, where $\cO_X(h):=\cO_{\p4}(1)\otimes\cO_X$ and $\ell$ is the class of a line. The class of a conic on $X$ is $h^2$: each conic in $X$ is the intersection of $X$ with a unique plane in $\p4$, hence $\Gamma(X)$ is the grassmannian $G(3,5)$ of planes in $\p4$. On $X$ there is also an indecomposable rank $2$ bundle $\cS$ called spinor bundle whose sections vanish exactly on the lines inside $X$. We have $c_1(\cS)=-h$, $c_2(\cS)=\ell$, $h^i\big(\cS(th)\big)=0$ for $t\in\bZ$ and $i=1,2$, $h^0\big(\cS\big)=0$, hence $\cS$ is $\mu$--stable. Moreover, $\cS$ fits into an exact sequence of the form
\begin{equation*}
0\longrightarrow\cS\longrightarrow\cO_X^{\oplus4}\longrightarrow\cS^\vee\longrightarrow0
\end{equation*}
(see \cite{Ott2} for all the details on spinor bundles on $X$). It follows that $\Lambda(X)\cong\p3$.

Let $\delta:=h^3\in\{\ 3,4,5\ \}$. Thus $X$ contains a line: if $\ell$ is its class in $H^4(X)$, then $h^2=\delta\ell$ and the class of a conic on $X$ is $2\ell$. In this case $\Lambda(X)$ and $\Gamma(X)$ are respectively a smooth surface and a smooth fourfold:  for more details see \cite[Propositions 2.2.10 and 2.3.8]{K--P--S} and the references therein.
\end{remark}

\begin{remark}
\label{rFlag}
Let $h^3=6$ and $\varrho_X=2$, i.e. $X=F_{6,2}$. 

If $\pi_i\colon F_{6,2}\to \p2$ denotes the projection to the $i^{\mathrm {th}}$ factor, then $\cO_{F_{6,2}}(h_i):=\pi_i^*\cO_{\p2}(1)$ is globally generated,  $h=h_1+h_2$ and
$$
H^*({F_{6,2}})\cong\frac{ \bZ[h_1,h_2]}{(h_1^2-h_1h_2+h_2^2,h_1^3,h_2^3)}.
$$
The fibres of $\pi_i$ are lines with class $h_i^2$. Conversely, let $\beta_1h_2^2+\beta_2h_1^2\in H^4({F_{6,2}})$ be the class of a line $L$.  Thus $\beta_1+\beta_2=hL=1$ and $\beta_i=Lh_i\ge0$, hence the class of $L$ is $h_i^2$ and $L$ is a fibre of $\pi_i$. In particular two distinct lines in the same class do not intersect.  We conclude that $\Lambda({F_{6,2}})$ has two components both isomorphic to $\p2$.

Now let $\beta_1h_2^2+\beta_2h_1^2\in H^4({F_{6,2}})$ be the class of a conic $C$, i.e. a curve of degree $2$ with $p_a(C)=0$. Thus $\beta_1+\beta_2=hC=2$ and $\beta_i=Ch_i\ge0$. We then infer that the class of $C$ is either $h_1h_2$ (coinciding with $h_1^2+h_2^2$ as cycles) or $2h_i^2$. 

Let us consider the latter case. If $C$ is reducible, then it is the union of two fibres of $\pi_i$ which are skew lines, hence $p_a(C)=-1$, a contradiction. Thus $C$ must be a double structure on a line $L$ with class $h_i^2$. Since $h_i^3=0$, it follows that each divisor in $\vert h_i\vert$ intersecting $L$ actually contains it, hence there is an exact sequence 
\begin{equation}
\label{seqLineFlag}
0\longrightarrow\cO_{F_{6,2}}(-2h_i)\longrightarrow\cO_{F_{6,2}}(-h_i)^{\oplus2}\longrightarrow\cI_{L\vert F_{6,2}}\longrightarrow0.
\end{equation}
We have $\cI_{L\vert F_{6,2}}^2\subseteq\cI_{C\vert F_{6,2}}\subseteq\cI_{L\vert F_{6,2}}$, hence a surjective morphism 
$$
\nu\colon \cN_{L\vert F_{6,2}}^\vee\cong\cI_{L\vert F_{6,2}}/\cI_{L\vert F_{6,2}}^2\longrightarrow \cI_{L\vert F_{6,2}}/\cI_{C\vert F_{6,2}}\cong\cI_{L\vert C}.
$$
On the one hand, Sequence \eqref{seqLineFlag} restricted to $\cO_L$ yields $\cN_{L\vert F_{6,2}}\cong\cO_{\p1}^{\oplus2}$. On the other hand, the latter sheaf is the kernel of the natural map $\cO_C\to\cO_L$, hence its Hilbert polynomial is $t$. Since $\cI_{L\vert F_{6,2}}^2\subseteq\cI_{C\vert F_{6,2}}$, it follows that $\cI_{L\vert C}$ is a sheaf of rank $1$ on $L$. We deduce that $\cI_{L\vert C}\cong\cO_{\p1}(-1)$, hence the morphism $\nu$ cannot exist. In particular, this case cannot occur.

Thus the class of $C$ is necessarily $h_1h_2$. If $C$ is integral, then the equalities $h_1h_2h_i=1$  and $h^0\big(\cO_{F_{6,2}}(h_i)\big)=3$ imply that $C$ is contained in the complete intersection of two divisors in $\vert h_1\vert$ and $\vert h_2\vert$, hence  it must coincide with it by degree reasons. This implies the existence of an exact sequence of the form
\begin{equation}
\label{seqFlag}
0\longrightarrow\cO_{F_{6,2}}(-h)\longrightarrow\cO_{F_{6,2}}(-h_1)\oplus\cO_{F_{6,2}}(-h_2)\longrightarrow\cI_{C\vert {F_{6,2}}}\longrightarrow0.
\end{equation}
If $C$ is not integral, then it must be the union of a fibre $L_1$ of $\pi_1$ and a fibre $L_2$ of $\pi_2$. In particular there  are pencils $\Sigma_i\subseteq\vert h_i\vert$ with base locus $L_i$. Thus we can find a divisor inside $\Sigma_1$ containing a second point in $L_2$ besides $L_1\cap L_2$, hence the whole $C$ because $h_1h_2^2=1$. Similarly there is a divisor in $\Sigma_2$ containing $C$ and we get the above sequence again.

Notice that the Cohomology of Sequence \eqref{seqFlag} tensored by $\cO_{F_{6,2}}(h_i)$ yields the uniqueness of the divisors in $\vert h_i\vert$ containing $C$. In particular, $\Gamma({F_{6,2}})$ is $\vert h_1\vert\times\vert h_2\vert\cong\p2\times\p2$.
\end{remark}

\begin{remark}
\label{rSegre}
Let $h^3=6$ and $\varrho_X=3$, i.e. $X=F_{6,3}$. 

If $\pi_i\colon {F_{6,3}}\to \p1$ denote the $i^{\mathrm {th}}$--projection, then $\cO_{F_{6,3}}(h_i):=\pi_i^*\cO_{\p1}(1)$ is globally generated,  $h=h_1+h_2+h_3$ and
$$
H^*({F_{6,3}})\cong\frac{ \bZ[h_1,h_2,h_3]}{(h_1^2,h_2^2,h_3^2)}.
$$
The fibres of $\pi_i$ are smooth quadrics with class $h_i$. Thus the fibres of $\pi_i\times\pi_j$ have class $h_ih_j$, hence they are lines. Conversely, let $\beta_1h_2h_3+\beta_2h_1h_3+\beta_3h_1h_2\in A^2({F_{6,3}})$ be the class of a line $L$.  Thus $\beta_1+\beta_2+\beta_3=hL=1$ and $\beta_i=Lh_i\ge0$, hence the class of $L$ is $h_ih_j$ and $L$ is a fibre of $\pi_i\times\pi_j$. In particular two lines in the same class do not intersect. In particular, $\Lambda({F_{6,3}})$ has three components isomorphic to $\p1\times\p1$.

Now let $\beta_1h_2h_3+\beta_2h_1h_3+\beta_3h_1h_2\in A^2({F_{6,3}})$ be the class of a conic $C$. Thus $\beta_1+\beta_2+\beta_3=hC=2$ and $\beta_i=Ch_i\ge0$. Up to permutations of the indices, we then infer that the class of $C$ is either $h_1(h_2+h_3)$, or $2h_1h_2$. As in Remark \ref{rFlag}, one can exclude the latter case and find an exact sequence of the form
\begin{equation}
\label{seqSegre}
0\longrightarrow\cO_{F_{6,3}}(-h)\longrightarrow\cO_{F_{6,3}}(-h_1)\oplus\cO_{F_{6,3}}(-h_2-h_3)\longrightarrow\cI_{C\vert {F_{6,3}}}\longrightarrow0.
\end{equation}

As in Remark \ref{rFlag} computing the Cohomology of Sequence \eqref{seqSegre} tensored by $\cO_{F_{6,3}}(h_1)$ and $\cO_{F_{6,3}}(h_2+h_3)$ one checks that $\Gamma({F_{6,3}})\cong\p1\times\p3$.
\end{remark}

\begin{remark}
\label{rBlow}
Let $h^3=7$, hence $X=F_{7}$. 

We have two natural projections, i.e. the blow up map $\sigma\colon {F_7}\to \p3$ and the map $\pi\colon {F_7}\to \p2$  induced by the projection from $p$ onto a plane $\p2$: we define $\cO_{F_7}(\xi):=\sigma^*\cO_{\p3}(1)$ and $\cO_{F_7}(f):=\pi^*\cO_{\p2}(1)$. Notice that both $\cO_{F_7}(\xi)$ and $\cO_{F_7}(f)$ are globally generated. The map $\pi$ induces an isomorphism ${F_7}\cong\bP(\cO_{\p2}\oplus\cO_{\p2}(1))$, $h=\xi+f$ and
$$
H^*({F_7})\cong\frac{ \bZ[\xi,f]}{(f^3,\xi^2-\xi f)}.
$$
The fibres of $\pi$ are lines with class $f^2$. Moreover, ${F_7}$ also contains the exceptional divisor $E\cong\p2$ of $\sigma$. We have $\{\ E\ \}=\vert \xi-f\vert$, hence $E$ is also embedded in $\p8$ as a plane, because $Eh^2=1$. Trivially each line inside $E$ is also a line in ${F_7}$ whose image via $\pi$ is a line. Conversely, arguing as in the previous remarks one deduces that all the lines on $X$ are of the above form. In particular, $\Lambda({F_7})$ has two components both isomorphic to $\p2$.

Now let $\beta_1\xi f+\beta_2f^2\in A^2({F_7})$ be the class of a conic $C$. Thus $2\beta_1+\beta_2=hC=2$, $\beta_1=Cf\ge0$ and $\beta_1+\beta_2=C\xi\ge0$. It is easy to check that the class of $C$ is $\xi f$, or $2(\xi-f)f$, or $f^2$ and one can exclude the latter case as in Remark \ref{rFlag}. Notice that in the first case $C$ does not intersect $E$, while in the second $C\subseteq E$. Again as in Remark \ref{rFlag} one obtains the existence of exact sequences 
\begin{equation}
\label{seqBlow1}
0\longrightarrow\cO_{F_7}(-h)\longrightarrow\cO_{F_7}(-\xi)\oplus\cO_{F_7}(-f)\longrightarrow\cI_{C\vert {F_7}}\longrightarrow0,
\end{equation}
if the class of $C$ is $\xi f$, and 
\begin{equation}
\label{seqBlow2}
0\longrightarrow\cO_{F_7}(-h)\longrightarrow\cO_{F_7}(-\xi+f)\oplus\cO_{F_7}(-2f)\longrightarrow\cI_{C\vert {F_7}}\longrightarrow0,
\end{equation}
if the class of $C$ is $2(\xi-f)f$.

Thus $\Gamma({F_7})$ is the union of two components. One of them consists of conics in $E\cong\p2$, hence it is isomorphic to $\p5$. The other one is given by the complete intersections of divisors in $\vert f\vert$ and $\vert \xi\vert$. As in Remark \ref{rFlag}, one deduces from the analysis of Sequences \eqref{seqBlow1} and \eqref{seqBlow2} that the divisor in $A\in\vert f\vert$ through $C$ is unique and that $h^0\big(\cO_{F_7}(\xi)\otimes\cO_A\big)=3$, hence the second component is a projective bundle over $\vert f\vert\cong\p2$ with fibre $\p2$. In particular it has dimension $4$.
\end{remark}

We finally prove the main theorem of the section.

\medbreak
\noindent{\it Proof of Theorem \ref{tSharp}.}
We know that $h^0\big(\cE\big)=0$ by definition. Moreover, if $\cE$ is non--ordinary, then $q_X^\varepsilon\ge1$. Table 1 will be sometimes helpful in what follows.
\begin{table}[H]
\label{TableInvariants}
\centering
\bgroup
\def\arraystretch{1.5}
\begin{tabular}{ccccccccc}
\cline{1-9}
\multicolumn{1}{|c|}{$i_X$} &\multicolumn{1}{c}{4} & \multicolumn{1}{c}{4} & \multicolumn{1}{c}{3} & \multicolumn{1}{c}{3} & \multicolumn{1}{c}{2} & \multicolumn{1}{c}{2}& \multicolumn{1}{c}{1}& \multicolumn{1}{c|}{1}\\ 
\cline{1-9}

\multicolumn{1}{|c|}{$\varepsilon$} &\multicolumn{1}{c}{0} & \multicolumn{1}{c}{1} & \multicolumn{1}{c}{0} & \multicolumn{1}{c}{1} & \multicolumn{1}{c}{0} & \multicolumn{1}{c}{1}& \multicolumn{1}{c}{0}& \multicolumn{1}{c|}{1}\\ 
\cline{1-9}

\multicolumn{1}{|c|}{$ i_X-\varepsilon$} &\multicolumn{1}{c}{4} & \multicolumn{1}{c}{3} & \multicolumn{1}{c}{3} & \multicolumn{1}{c}{2} & \multicolumn{1}{c}{2} & \multicolumn{1}{c}{1}& \multicolumn{1}{c}{1}& \multicolumn{1}{c|}{0}\\
\cline{1-9}

\multicolumn{1}{|c|}{$q_X^\varepsilon$}&\multicolumn{1}{c}{2} & \multicolumn{1}{c}{2} & \multicolumn{1}{c}{2} & \multicolumn{1}{c}{1} & \multicolumn{1}{c}{1} & \multicolumn{1}{c}{1}& \multicolumn{1}{c}{1}& \multicolumn{1}{c|}{0}\\ 
\cline{1-9}

\multicolumn{1}{|c|}{ordinary}&\multicolumn{1}{c}{yes} & \multicolumn{1}{c}{no} & \multicolumn{1}{c}{no} & \multicolumn{1}{c}{yes} & \multicolumn{1}{c}{yes} & \multicolumn{1}{c}{no}& \multicolumn{1}{c}{no}& \multicolumn{1}{c|}{yes}\\ 
\cline{1-9}

\end{tabular}
\egroup
\caption{Values of $i_X$, $\varepsilon$, $i_X-\varepsilon$, $q_X^\varepsilon$.}
\end{table}

We have $h^i\big(\cE\big)=h^{3-i}\big(\cE((\varepsilon-i_X)h)\big)$, thanks to Equality \eqref{SerreDual}. 
One  checks directly that $\varepsilon-i_X\le -q_X^\varepsilon\le0$, hence $h^3\big(\cE\big)=h^0\big(\cE((\varepsilon-i_X)h)\big)=0$. Lemma \ref{lH12} implies  $h^2\big(\cE\big)=0$ if $i_X-\varepsilon\ne0$, i.e. if $(\varepsilon,i_X)\ne(1,1)$.

Thanks to Equality \eqref{RRGeneral}, when $(\varepsilon,i_X)\ne(1,1)$, we then obtain
\begin{equation}
\label{LowerBound}
\varepsilon-2+\frac{i_X-\varepsilon}2hc_2(\cE)=-\chi(\cE)=h^1\big(\cE\big)\ge0.
\end{equation}
Thus $hc_2(\cE)$ must be even and positive when $i_X-\varepsilon$ is odd, i.e. when $\cE$ is non--ordinary. 
Moreover, in this case $i_X-\varepsilon\ne0$, hence Inequality \eqref{LowerBound} returns
\begin{equation*}
hc_2(\cE)\ge\left\lbrace\begin{array}{ll} 
2\quad&\text{if $i_X\ge2$,}\\
4\quad&\text{if $i_X=1$.}
\end{array}\right.
\end{equation*}

Thus, in order to complete the proof of the statement, we have only to check that there is no non--ordinary indecomposable instanton bundle $\cE$ such that $c_2(\cE)h=2$ when $i_X=2$ and $h^3\ge6$. If it exists, then $h^2\big(\cE(h)\big)=0$ thanks to Lemma \ref{lH12}. It follows that $h^0\big(\cE(h)\big)\ge\chi(\cE(h))=h^3$, thanks to Equality \eqref{RRGeneral}. Let $s\in H^0\big(\cE(h)\big)$ be a non--zero section. 

We first assume that the zero--locus  of the section $s$ is the union of a curve $Z$ and an effective divisor $D\subseteq X$. Thus, there is a non--zero section of $\cE(h-D)$ vanishing exactly along $Z$ whose class inside $A^2(X)$ is $c_2(\cE(h-D))=c_2(\cE)-hD+D^2$. We know that
$$
Dh^2-h^3=(D-h)h^2=\mu(\cO_X(D-h))\le \mu(\cE)=-\frac {h^3}2,
$$
because $\cE$ is $\mu$--semistable, thanks to Lemma \ref{lHoppe}. We deduce that
\begin{equation}
\label{BoundDivisor}
2Dh^2\le h^3.
\end{equation}

Let $X=F_{6,2}$. If $D\in\vert \alpha_1h_1+\alpha_2h_2\vert$, then $\alpha_1=Dh_2^2\ge0$, because $D$ is effective: similarly $\alpha_2\ge0$. Inequality \eqref{BoundDivisor} yields $\alpha_1+\alpha_2\le1$, hence without restricting we can then assume $D=h_1$. It then follows  $hc_2(\cE(h-D))=0$, hence $Z=\emptyset$ and the Koszul complex of $s$ becomes 
$$
0\longrightarrow\cO_{F_{6,2}}\longrightarrow\cE(h_2)\longrightarrow\cO_{F_{6,2}}(h_2-h_1)\longrightarrow0.
$$
Such an exact sequence splits, because $\Ext^1_{F_{6,2}}\big(\cO_{F_{6,2}}(h_2-h_1),\cO_{F_{6,2}}\big)=0$ (see \cite[Proposition 2.5]{C--Fa--M2}), hence we infer $\cE\cong\cO_{F_{6,2}}(-h_1)\oplus\cO_{F_{6,2}}(-h_2)$. 

Let $X={F_{6,3}}$. Arguing as in the previous case, we can still assume that $D=h_1$ and we still deduce $Z=\emptyset$. The cohomology of the Koszul complex of $s$ and the K\"unneth formulas imply $h^0\big(\cE(h_1)\big)\ne0$, contradicting the $\mu$--semistability of $\cE$, hence this case cannot occur.

Let $X=F_7$. If $D\in\vert\alpha_1\xi+\alpha_2 f\vert$, then $\alpha_1,\alpha_1+\alpha_2\ge0$, because $D$ is effective, and Inequality \eqref{BoundDivisor} yields $8\alpha_1+6\alpha_2\le7$. It follows that either $D=f$ or $D=\xi-f$ and we can easily deduce as in the previous case that such a case cannot occur too.

We now assume that the zero--locus of $s$ is a curve $Z$. The degree of $Z$ is $2$: the Koszul complex of $s$ and the inequality $h^0\big(\cE(h)\big)\ge h^3$ imply that $Z$ is contained in a plane. We deduce that $Z$ is a possibly non--integral conic, hence $\omega_Z\cong\cO_{X}(-h)\otimes\cO_Z$. In what follows we will again make a case by case analysis.

If  $X={F_{6,2}}$, then we have Sequence \eqref{seqFlag}. Recall that we also have  the Kosxul complex of $s$, i.e.
$$
0\longrightarrow\cO_{F_{6,2}}\longrightarrow\cE(h)\longrightarrow\cI_{Z\vert {F_{6,2}}}(h)\longrightarrow0.
$$
Equality \eqref{SerreDual} and the cohomology of Sequence \eqref{seqStandard}
tensored by $\cO_{F_{6,2}}(-h)$ yield
\begin{align*}
\Ext^1_{F_{6,2}}\big(\cI_{Z\vert {F_{6,2}}}(h),\cO_{F_{6,2}})^\vee&\cong \Ext^2_{F_{6,2}}\big(\cO_{F_{6,2}},\cI_{Z\vert {F_{6,2}}}(-h)\big)\\
&\cong H^1\big(\cO_{F_{6,2}}(-h)\otimes\cO_Z\big)\cong H^0\big(\cO_Z\big)\cong\bC.
\end{align*}
Taking account of the equality $h^1\big(\cO_{F_{6,2}}(-h)\big)=0$, Theorem \ref{tSerre} and Sequence \ref{seqFlag}, we deduce that $\cE\cong\cO_{F_{6,2}}(-h_1)\oplus\cO_{F_{6,2}}(-h_2)$ taking into account Sequence \ref{seqFlag}, because there is only one non--trivial extension of $\cI_{Z\vert {F_{6,2}}}(h)$ by $\cO_X$.

We exclude the two remaining cases again using the $\mu$--semistability of $\cE$. 
If $X={F_{6,3}}$ and $Z$ is the complete intersection of divisors in $\vert h_1\vert$ and $\vert h_2+h_3\vert$, the cohomology of Sequence \eqref{seqSegre} tensored by $\cO_{F_{6,3}}(h_1)$ yields
$h^0\big(\cE(h_1)\big)=h^0\big(\cI_{Z\vert {F_{6,3}}}(h_1)\big)\ne0$.
Since $\mu(\cO_{F_{6,3}}(-h_1))=-2>-3=\mu(\cE)$, it follows from Lemma \ref{lHoppe} that this case does not occur. 

A similar argument can be applied also to the case $h^3=7$, i.e. $X:={F_7}$, by computing the cohomology of Sequences \eqref{seqBlow1} and \eqref{seqBlow2} respectively tensored by $\cO_{F_7}(f)$ and $\cO_{F_7}(\xi-f)$.
\qed
\medbreak

\begin{remark}
\label{rSplit}
As shown in the above proof, when $i_X=2$ and $\varrho_X\ge2$ the only  instanton bundle with quantum number $\quantum=2$ is  $\cO_{F_{6,2}}(-h_1)\oplus\cO_{F_{6,2}}(-h_2)$.
\end{remark}

\section{Non--ordinary instanton bundles when $i_X\ge2$}
\label{sExistence}
The existence of ordinary instanton bundles on a Fano threefold $X$ with very ample $\cO_X(h)$ has been the object of several papers. More precisely ordinary instanton bundles always exist when $i_X\ge2$ and in several cases when $i_X=1$ (see the introduction and the references therein for further details). 

In this section we will show that $X$ also supports non--ordinary instanton bundles when $i_X\ge2$. The first step in this direction is the following result.

\begin{lemma}
Let $X$ be a Fano threefold with $i_X\ge2$ and very ample $\cO_X(h)$.

For each non--negative integer $s$, there is a curve $Z\subseteq X$ whose connected components are  $s$ pairwise disjoint integral conics.
\end{lemma}
\begin{proof}
The statement is trivial if $i_X\ge3$, hence we only need to examine the case $i_X=2$ in what follows. Let $d:=h^3$ be the degree of $X$: we know that $3\le d\le 7$. 

In this case the general hyperplane section $H$ of $X$ is a del Pezzo surface obtained by blowing up $r=9-d$ points in general position on $\p2$. Thus $\Pic(H)$ is generated by the pull--back $\ell$ of the class of a general line in $\p2$ and by the classes $e_1,\dots ,e_r$ of the exceptional divisors. Moreover, $\ell^2=-e_i^2=1$, $\ell e_i=e_ie_j=0$ for $1\le i< j\le r$. 

Since the class of the hyperplane section of $H$ is $3\ell-\sum_{k=1}^re_k$, it follows that that every element in $\vert \ell-e_1\vert$ is a conic. Notice that two such conics do not intersect.
\end{proof}

Now let us fix a Fano threefold $X$ with $i_X\ge2$ and consider $\varepsilon\in\{\ 0,1\ \}$ such that $i_X-\varepsilon$ is odd. The following construction extends \cite[Example 3.1.2]{Ha4} to all Fano threefolds with $i_X\ge2$.

\begin{construction}
\label{conInstanton}
Let $X$ be a Fano threefold  with very ample $\cO_X(h)$.
For each even integer $\quantum\ge2$, we consider the union $Z$ of $s:=(\quantum+2q_X^\varepsilon-2)/2$ pairwise disjoint integral conics on $X$. The adjunction formula on $X$ and the definition of $Z$ imply
$\det(\cN_{Z\vert X})\cong\cO_X((i_X-1)h)\otimes\cO_Z$.
Since $h^i\big(\cO_X(th)\big)=0$ for $i=1,2$ and each $t\in\bZ$, it follows the existence of an exact sequence 
\begin{equation}
\label{seqF}
0\longrightarrow\cO_X\longrightarrow\cF\longrightarrow\cI_{Z\vert X}((i_X-1)h)\longrightarrow0,
\end{equation}
thanks to Theorem \ref{tSerre}. The definition of $q_X^\varepsilon$ implies $i_X-1=2(q_X^\varepsilon-1)+\varepsilon$, hence the above sequence yields
\begin{equation}
\label{seqInstanton}
0\longrightarrow\cO_X((1-q_X^\varepsilon-\varepsilon)h)\longrightarrow\cE\longrightarrow\cI_{Z\vert X}((q_X^\varepsilon-1)h)\longrightarrow0.
\end{equation}
\end{construction}

On the one hand, if $i_X=1$, then the above construction does not lead to any instanton bundle $\cE$ for every value of $\quantum$, because $h^0\big(\cE\big)=1$ in this case. On the other hand, in what follows we will show that the construction above leads to non--ordinary instanton bundles for all the admissible values of the quantum number $\quantum$ given in Bound \eqref{BoundSharp} when $i_X\ge2$. 

We are now ready to prove Theorem \ref{tExistence}. We first start by proving that all the admissible values of the quantum number $\quantum$ can be obtained when $\varrho_X=1$.

\medbreak
\noindent{\it Proof of Theorem \ref{tExistence} when $\varrho_X=1$.}
Since $\varrho_X=1$ we know that each divisor $D$ on $X$ is linearly equivalent to $dh$. Moreover, such a $D$ is effective if and only if $d\ge1$.

It is immediate from Construction \ref{conInstanton} that $c_1(\cE)=-\varepsilon h$ and $hc_2(\cE)=\quantum$. The cohomology of Sequence \eqref{seqInstanton} implies 
$h^0\big(\cE\big)\le h^0\big(\cI_{Z\vert X}((q_X^\varepsilon-1)h)\big)$.
The above vanishing is trivial if $q_X^\varepsilon=1$. Thus we have to deal with $\p3$ and the smooth quadric $Q\subseteq\p4$: in this case $q_X^\varepsilon=2$. 

In both cases $Z$ is the union of $s=(\quantum+2)/2\ge2$ pairwise disjoint conics. Thus, if $X\cong\p3$, then $Z$ is not contained in any plane, hence $h^0\big(\cI_{Z\vert X}(h)\big)=0$. Assume $X\cong Q$ and let $h^0\big(\cI_{Z\vert X}(h)\big)\ne0$. Each hyperplane $H\subseteq\p4$ intersects $Q$ along a quadric $Q_H\subseteq H\cong\p3$ which has rank at least $3$, hence it is integral as well. It is easy to check that two conics on an integral quadric in $\p3$ always intersect. It follows that a $h^0\big(\cI_{Z\vert X}(h)\big)=0$ necessarily. Consequently, $h^0\big(\cE\big)=0$ also when $X$ is either $\p3$ or $Q$.

Let $D$ be a divisor on $X$ such that $\mu(\cO_X(D))\ge\mu(\cE)$. If $dh$ is its class, then such a condition forces $d\ge-\varepsilon/2\ge-1/2$, hence $d\ge0$ necessarily. Thus $h^0\big(\cE(-D)\big)\le h^0\big(\cE\big)=0$. Thanks to Lemma \ref{lHoppe}, we deduce that $\cE$ is $\mu$--stable. 
\qed
\medbreak

We now deal with the case $\varrho_X\ge2$ distinguishing two cases.

\medbreak
\noindent{\it Proof of Theorem \ref{tExistence} when $\varrho_X\ge2$.}
The proof is similar to the previous one. We have only to put some more care in showing that $\cE$ is $\mu$--(semi)stable.

We have three distinct cases to deal with. The threefold $X$ can be $F_{6,2}\subseteq\p8$, or $F_{6,3}\subseteq\p7$, or $F_7\subseteq\p8$.

Let $X:= F_{6,2}$. Recall that conics on ${F_{6,2}}$ were described in Remark \ref{rFlag}, where we showed that $\Gamma({F_{6,2}})\cong\vert h_1\vert\times\vert h_2\vert$. Thus two distinct components of $Z$ cannot be contained in the same divisor in either $\vert h_1\vert$ or $\vert h_2\vert$.

We show that $\cE$ is $\mu$--(semi)stable using Lemma \ref{lHoppe} and \cite[Proposition 2.5]{C--Fa--M2}. To this purpose let $\cO_{F_{6,2}}(D)\cong\cO_{F_{6,2}}(\alpha_1h_1+\alpha_2 h_2)$: we have $\mu(\cO_{F_{6,2}}(D))\ge\mu(\cE)$ if and only if $\alpha_1+\alpha_2\ge-1$. Thus either both $\alpha_i$'s are non--negative, or one of them is strictly negative and the other one non--negative. In both cases $h^0\big(\cO_{F_{6,2}}(-h-D)\big)=0$ thanks to \cite[Proposition 2.5]{C--Fa--M2}. The cohomology of Sequence \eqref{seqInstanton} yields 
\begin{equation}
\label{BoundStable}
h^0\big(\cE(-D)\big)\le h^0\big(\cI_{Z\vert{F_{6,2}}}(-D)\big).
\end{equation}
If $\alpha_i\ge1$ for some $i$, then the dimension on the right--hand side of Inequality \eqref{BoundStable} vanishes. Thus we must deal with the case $\alpha_1,\alpha_2\le0$. If equality holds, then again the dimension on the right vanishes. In particular $\cE$ is certainly at least $\mu$--semistable. It then remains to deal with the case $-D=h_i$ for some $i$. In this case \cite[Proposition 2.5]{C--Fa--M2} yields $h^1\big(\cO_{F_{6,2}}(-h+h_i)\big)=0$, hence $h^0\big(\cE(h_i)\big)= h^0\big(\cI_{Z\vert {F_{6,2}}}(h_i)\big)$. Since  $k\ge4$, it follows that $Z$ has at least two connected components: as we showed above such components cannot be contained in the same divisor in $\vert h_i\vert$, hence $h^0\big(\cI_{Z\vert {F_{6,2}}}(h_i)\big)=0$. It follows that $\cE$ is $\mu$--stable in this range, hence it is also indecomposable.

Let $X:=F_{6,3}$. Recall that each irreducible component in $Z$ is the complete intersection of divisors in $\vert h_i\vert$ and $\vert h-h_i\vert$ and that two distinct components of $Z$ cannot be contained in the same divisors in either $\vert h_i\vert$ or $\vert h-h_i\vert$. We will deal with the $\mu$--semistability of $\cE$ as in the previous case using Lemma \ref{lHoppe} and K\"unneth formulas. In this case if $\cO_{F_{6,3}}(D)\cong\cO_{F_{6,3}}(\alpha_1h_1+\alpha_2 h_2+\alpha_3h_3)$ we have $\mu(\cO_{F_{6,3}}(D))\ge\mu(\cE)$ if and only if  $\alpha_1+\alpha_2+\alpha_3\ge-3/2$. Thus either all the $\alpha_i$'s are non--negative, or one of them is strictly negative and one non--negative. In both cases $h^0\big(\cO_{F_{6,3}}(-h-D)\big)=0$, hence we can use the same argument used for proving the $\mu$--stability in the previous case. 

Finally we examine the $\mu$--stability of $\cE$ when $X:=F_7$ using Lemma \ref{lHoppe} and \cite[Proposition 2.7]{C--Fi--M}. The argument is essentially the same used before, but we need some more care because the geometry of ${F_7}$ is slightly richer. Indeed in this case $\Gamma({F_7})$ is the union of the fourfold of conics which are complete intersection of divisors in $\vert \xi\vert$ and $\vert f\vert$, with the $\p5$ of conics inside $E$, as we showed in Remark \ref{rBlow}.  Anyhow there are still no divisors in $\vert f\vert$ containing more than one component of $Z$. In this case if $\cO_{F_7}(D)\cong\cO_{F_7}(\alpha_1\xi+\alpha_2f)$, then $\mu(\cO_{F_7}(D))\ge\mu(\cE)$ if and only if
\begin{equation}
\label{BoundBlowUp}
4\alpha_1+3\alpha_2\ge-7/2.
\end{equation}
The cohomology of Sequence \eqref{seqInstanton} yields 
\begin{align*}
h^0\big(\cE(-\alpha_1\xi-\alpha_2f)\big)&\le h^0\big(\cO_{{F_7}}(-(\alpha_1+1)\xi-(\alpha_2+1)f)\big)\\
&+h^0\big(\cI_{Z\vert {F_7}}(-\alpha_1\xi-\alpha_2f)\big).
\end{align*}
If $\alpha_1\ge1$ the last member is zero, because $-\alpha_1=(-\alpha_1\xi-\alpha_2f)f^2$. Thus we have to deal with the case $\alpha_1\le0$. In this case Inequality \eqref{BoundBlowUp} yields
$$
-(\alpha_1+1)-(\alpha_2+1)\le -\alpha_1+\frac{8\alpha_1+7}6-2=\frac{2\alpha_1-5}6<0
$$
hence $h^0\big(\cO_{{F_7}}(-(\alpha_1+1)\xi-(\alpha_2+1)f)\big)=0$. Moreover,
$$
-\alpha_1-\alpha_2\le -\alpha_1+\frac{8\alpha_1+7}6=\frac{2\alpha_1+7}6.
$$
If $-\alpha_1-\alpha_2\le-1$, then $h^0\big(\cI_{Z\vert {F_7}}(-\alpha_1\xi-\alpha_2f)\big)\le h^0\big(\cO_{{F_7}}(-\alpha_1\xi-\alpha_2f)\big)=0$. 
Thus we have to check the cases $-\alpha_1-\alpha_2=1$ and $-\alpha_1-\alpha_2=0$. 

In the former case $\alpha_2=-1-\alpha_1$, hence Inequality \eqref{BoundBlowUp} returns $2\alpha_1\ge-1$. The hypothesis $\alpha_1\le0$ then forces $\alpha_1=0$ and we have to check if $h^0\big(\cI_{Z\vert {F_7}}(f)\big)=0$. In the latter case $\alpha_2=-\alpha_1$ and we have to check if $h^0\big(\cI_{Z\vert {F_7}}(-\alpha_1\xi+\alpha_1f)\big)=0$. This is true if $\alpha_1=0$, hence we consider the case $\alpha_1\le -1$ in what follows. In this case $h^0\big(\cO_{{F_7}}(-\alpha_1\xi+\alpha_1f)\big)=1$, hence the linear system $\vert -\alpha_1\xi+\alpha_1f\vert$ only contains the divisor $-\alpha_1E$. Since we are assuming $\quantum\ge4$, then $Z$ has at least two connected components, which cannot be contained simultaneously in either a divisor in $\vert f\vert$ or in $E$, as explained above. Thus the proof of the $\mu$--stability of $\cE$ is complete.
\qed
\medbreak

\section{Non--ordinary instanton bundles on ordinary prime Fano threefolds}
\label{sPrime}
In this section we will give some results on instanton bundles on an ordinary prime Fano threefold $X$, proving Theorem \ref{tPrime}. We will also assume that $\cO_X(h)$ is very ample. 

In this case $H^2(X)$ and $H^4(X)$ are respectively generated by the hyperplane class $h$ and by the class $\ell$ of a line. Equality \eqref{RRGeneral} returns 
$$
\chi(\cO_X(h))=\frac{h^3}2+3.
$$
Thus $\deg(X)=h^3$ is an even integer. The number $g_X:=\chi(\cO_X(h))-2$ is usually called the {\sl genus of $X$}. If $\varrho_X=1$, then $3\le g_X\le 12$ and $g_X\ne11$. 

In this case the irreducible components of $\Lambda(X)$ and $\Gamma(X)$ have dimensions $1$ and $2$ respectively (see \cite[Lemmas 2.2.3 and 2.3.4]{K--P--S}). Moreover, there is component $\Gamma(X)_0\subseteq \Gamma(X)$ whose general point is an integral conic $Z$ by \cite[Theorem 4.4]{Isk2}. The conics in $\Gamma(X)_0$ cover $X$ by \cite[Lemma 2.3.4]{K--P--S}, hence the normal bundle of $Z$ satisfies $\cN_{Z\vert X}\cong\cO_\p1^{\oplus2}$ by \cite[Proposition  4.3]{Isk2}.

\begin{lemma}
\label{lConic}
Let $X$ be an ordinary prime Fano threefold  with very ample $\cO_X(h)$.

If $Z\in \Gamma(X)_0$ is general, then the general conic in $\Gamma(X)_0$ and the general line in each component of $\Lambda (X)$ do not intersect $Z$.
\end{lemma}
\begin{proof}
There exists a conic $C$ not intersecting $Z$: see \cite[Proof of Step 1 of the proof of Theorem 4.1]{B--F1} where it suffices to replace $\mathcal H^0_2(X)$ with $\Gamma(X)_0$. Thus the same is true for the general conic in $\Gamma(X)_0$.

Let us now prove that there is a also a line in each component of $\Lambda(X)_i\subseteq\Lambda(X)$ not intersecting $Z$. Assume that all the conics in $\Gamma(X)_0$ intersect each line in $\Lambda(X)_i$. The scheme $R_i:=\bigcup_{L\in \Lambda(X)_i}L\subseteq X$ is an irreducible surface. If $Z\in\Gamma(X)_0$ is general, then $Z\cap R_i$ is a finite set of points, because $X$ is dominated by $\Gamma(X)_0$. Since the surfaces $R_i$ are irreducible, it follows that we can find a point $v_i\in Z\cap R_i$ lying on all the lines in $R_i$. Thus each $R_i$ is a cone with vertex $v_i$ and the general  conic in $\Gamma(X)_0$ passes necessarily through all such points, hence it intersects $Z$, contradicting what we proved above.
\end{proof}

In what follows we want to construct instanton bundles on each prime Fano threefold. On the one hand, the only line bundle we can actually use in Serre construction for obtaining instanton bundles is $\cO_X$. Thus the only admissible construction by means of Serre correspondence is actually Construction \ref{conInstanton}. On the other hand, such a construction does not return instanton bundles $\cE$ because in this case $h^0\big(\cE\big)=1$, as pointed out in the previous section. For this reason we need a different and smarter approach for prime Fano threefolds.  

The existence of minimal instanton bundles on a Fano threefold $X$ follows from \cite[Theorem 4.1]{B--F1}. Thanks to an easy induction on the quantum number, it is possible to prove the existence of even instanton bundles on each prime Fano threefold for every admissible quantum number.

We will prove the existence of instanton bundles by induction on $\quantum$. The proof of the inductive step essentially coincides with few differences with the proof of \cite[Theorem 4.1]{B--F1} which suffices to prove the existence of minimal instanton bundles. Anyhow, in order to start with our induction we need some more properties of such minimal bundles. The proofs of these properties can be found in \cite{B--F1}, but they are a little hidden in the proofs and not explicitly indicated in the statements. Thus, for reader's benefit we explain how to recover them in what follows and how to adapt the proof in \cite{B--F1} to get the inductive step. 

To this purpose we introduce a new construction.

\begin{construction}
\label{conPrime}
Let $X$ be a prime Fano threefold  with very ample $\cO_X(h)$. 

If $\quantum=2$, let $\cE_2$ be the bundle $\cF$ obtained via Construction \ref{conInstanton} from a general  $Z\in\Gamma(X)_0$ and $\eta\colon \cO_X\to\cE_{2}$ the corresponding inclusion map. By construction the bundle $\cE_2$ is strictly $\mu$--semistable: moreover, restricting Sequence \eqref{seqF} to a general integral conic $Y\in\Gamma(X)_0$, we deduce that $\cE_2\otimes\cO_Y\cong\cO_\p1^{\oplus2}$.

If $\quantum\ge4$ is even, let $\cE_\quantum$ be a $\mu$--stable even instanton bundle with quantum number $\quantum$, such that $\cE_\quantum\otimes\cO_{Y}\cong\cO_\p1^{\oplus2}$ for  the general $Y\in\Gamma(X)_0$ and $\Ext_X^2\big(\cE_\quantum,\cE_\quantum\big)=0$. 

Thus, if $C\in\Gamma(X)_0$ is general, then the general $\varphi\colon \cE_\quantum\to\cO_C$ is surjective. We define $\cE_{\quantum,\varphi}:=\ker(\varphi)$.  By construction there is the exact sequence
\begin{equation}
\label{seqDeform}
0\longrightarrow\cE_{\quantum,\varphi}\longrightarrow\cE_\quantum\longrightarrow\cO_C\longrightarrow0.
\end{equation}
\end{construction}

While the existence of $\cE_2$ is clear, the existence of $\cE_\quantum$ when $\quantum\ge4$ will be proved in what follows. Anyhow $\cE_{\quantum,\varphi}$, if it exists, is not locally free at the points of $C$, thus it cannot be an instanton bundle, though it enjoys many properties of instanton bundles as we will show in the following proposition.

\begin{proposition}
\label{pKer}
Let $X$ be an ordinary prime Fano threefold  with very ample $\cO_X(h)$. 

For each even $\quantum\ge2$, the sheaf $\cE_{\quantum,\varphi}$ defined in Construction \ref{conPrime}  satisfies the following properties.
\begin{itemize}
\item $\cE_{\quantum,\varphi}$ is torsion--free and $c_1(\cE_{\quantum,\varphi})=0$, $hc_2(\cE_{\quantum,\varphi})=\quantum+2$, $c_3(\cE_{\quantum,\varphi})=0$.
\item $h^0\big(\cE_{\quantum,\varphi}\big)=h^1\big(\cE_{\quantum,\varphi}(-h)\big)=0$.
\item $\cE_{\quantum,\varphi}$ is simple.
\item For the general  $Y\in \Gamma(X)_0$, then $\cE_{\quantum,\varphi}\otimes\cO_Y\cong\cO_\p1^{\oplus2}$.
\item $\dim\Ext^1_{X}\big(\cE_{\quantum,\varphi},\cE_{\quantum,\varphi}\big)=2\quantum+1$, $\Ext^2_{X}\big(\cE_{\quantum,\varphi},\cE_{\quantum,\varphi}\big)=\Ext^3_{X}\big(\cE_{\quantum,\varphi},\cE_{\quantum,\varphi}\big)=0$.
\end{itemize}
\end{proposition}
\begin{proof}
Trivially $\cE_{\quantum,\varphi}$ is torsion free, because it is contained in the vector bundle $\cE_{\quantum}$. By combining Sequences \eqref{seqF} and \eqref{seqStandard} we obtain
\begin{equation}
\label{seqLong}
0\longrightarrow \cO_X\longrightarrow \cF\longrightarrow\cO_X \longrightarrow\cO_C \longrightarrow0.
\end{equation}
We deduce that the product of the Chern polynomials of $\cF$ and $\cO_C$ must be $1$, hence $c_1(\cO_C)=0$, $c_2(\cO_C)=-C$ and $c_3(\cO_C)=0$. Using these informations we compute the Chern classes of $\cE_{\quantum,\varphi}$ from Sequence \eqref{seqDeform} obtaining $c_1(\cE_{\quantum,\varphi})=0$, $c_2(\cE_{\quantum,\varphi})=c_2(\cE_{\quantum})+C$,  hence $hc_2(\cE_{\quantum,\varphi})=hc_2(\cE_{\quantum})+2$, and $c_3(\cE_{\quantum,\varphi})=0$.

When $\quantum\ge4$, the vanishings $h^0\big(\cE_{\quantum,\varphi}\big)=h^1\big(\cE_{\quantum,\varphi}(-h)\big)=0$ follow trivially from the cohomology of Sequence \eqref{seqDeform}, because the same vanishings hold for $\cE_\quantum$ by induction and $h^0\big(\cO_C\otimes\cO_X(-h)\big)=0$.

Let $\quantum=2$. Notice that $\im(\eta)\not\subseteq\cE_{2,\varphi}$, hence the restriction $\varphi\circ\eta\colon \cO_X\to\cO_{C}$ is surjective. Combining Sequences \eqref{seqF} and \eqref{seqDeform}, the Snake lemma implies the existence of an exact sequence
\begin{equation}
\label{seqBF}
0\longrightarrow\cI_{C\vert X}\longrightarrow\cE_{2,\varphi}\longrightarrow\cI_{Z\vert X}\longrightarrow0,
\end{equation}
whence we deduce $h^0\big(\cE_{2,\varphi}\big)=h^1\big(\cE_{2,\varphi}(-h)\big)=0$.

Let us prove that $\cE_{\quantum,\varphi}$ is simple. If $\quantum=2$ then it follows from \cite[Lemma 4.5]{B--F1}. If $\quantum\ge4$, then $\cE_{\quantum,\varphi}$ is actually $\mu$--stable, because each sheaf destabilizing $\cE_{\quantum,\varphi}$, also destabilizes $\cE_{\quantum}$, because $\mu(\cE_{\quantum,\varphi})=\mu(\cE_{\quantum})$. Thus $\cE_{\quantum,\varphi}$ is also simple thanks to \cite[Corollary 1.2.8]{H--L}.


By Lemma \ref{lConic}, a general $Y\in \Gamma(X)_0$ does not intersect $C$. Restricting Sequence \eqref{seqDeform} to $Y$ we then obtain  $\cE_{\quantum,\varphi}\otimes\cO_Y\cong\cE\otimes\cO_Y\cong\cO_\p1^{\oplus2}$. In the remaining part of the proof, we will compute  $\dim\Ext^i_{X}\big(\cE_{\quantum,\varphi},\cE_{\quantum,\varphi}\big)$, for $i\ge1$. 

Recall that $\Ext^3_{X}\big(\cE_{\quantum},\cE_{\quantum}\big)=0$ for each $\quantum\ge4$, thanks to Lemma \ref{lExt3}, because $\cE_\quantum$ is simple. If $\quantum=2$, then the isomorphism $\cE_2\cong\cE_2^\vee$ and Equality \eqref{SerreDual} yield
$$
\dim\Ext^3_{X}\big(\cE_{2},\cE_{2}\big)=h^0\big(\cE_2\otimes\cE_2(-h)\big).
$$
The cohomologies of Sequences \eqref{seqInstanton} and \eqref{seqStandard} tensored by $\cE_2(-h)$ imply
$$
h^0\big(\cE_2\otimes\cE_2(-h)\big)\le 2h^0\big(\cE_2(-h)\big).
$$
Again the cohomologies of the same sequences tensored by $\cO_X(-h)$ imply that the dimension on the right is zero.
Similarly, $\Ext^2_{X}\big(\cE_{\quantum},\cE_{\quantum}\big)=0$ when $\quantum\ge4$ by hypothesis. If $\quantum=2$, then $\Ext^2_{X}\big(\cE_{2},\cE_{2}\big)=0$ thanks to \cite[Lemma 4.4]{B--F1}. 

If we apply $\Hom_{X}\big(\cE_{\quantum},-\big)$ to Sequence \eqref{seqDeform}, then $\Ext^i_{X}\big(\cE_{\quantum},\cE_{\quantum,\varphi}\big)\cong\Ext^i_{X}\big(\cE_{\quantum},\cE_{\quantum}\big)$ for $i\ge2$, because $\Ext^i_{X}\big(\cE_{\quantum},\cO_C\big)\cong H^i\big(\cO_C^{\oplus2}\big)=0$ for $i\ge1$. It follows that $\Ext^3_{X}\big(\cE_{\quantum},\cE_{\quantum,\varphi}\big)=\Ext^2_{X}\big(\cE_{\quantum},\cE_{\quantum,\varphi}\big)=0$ for each $\quantum\ge2$. Thus, 
$$
\Ext^3_{X}\big(\cE_{\quantum,\varphi},\cE_{\quantum,\varphi}\big)=0, \qquad \Ext^2_{X}\big(\cE_{\quantum,\varphi},\cE_{\quantum,\varphi}\big)\subseteq\Ext^3_{X}\big(\cO_C,\cE_{\quantum,\varphi}\big),
$$
 for each $\quantum\ge2$, by applying $\Hom_{X}\big(-,\cE_{\quantum,\varphi}\big)$ to Sequence \eqref{seqDeform}.
 
Equality \eqref{SerreDual} implies $\Ext^3_{X}\big(\cO_C,\cE_{\quantum,\varphi}\big)\cong \Hom_{X}\big(\cE_{\quantum,\varphi},\cO_C\otimes\cO_X(-h)\big)$: notice that $\cO_C\otimes\cO_X(-h)\cong\cO_{\p1}(-2)$. Thus from $\cN_{C\vert X}\cong\cO_\p1^{\oplus2}$ we deduce
$$
\sExt^1_{X}\big(\cO_C,\cO_C\otimes\cO_X(-h)\big)\cong\cO_{\p1}(-2)^{\oplus2}
$$
(see Equalities \eqref{Local}). By the above equalities, applying $\sHom_{X}\big(-,\cO_{\p1}(-2)\big)$ to Sequence \eqref{seqDeform} we obtain the exact sequence
\begin{align*}
0\longrightarrow\cO_{\p1}(-2)\longrightarrow\cO_{\p1}(-2)^{\oplus2}\longrightarrow\sHom_{X}\big(\cE_{\quantum,\varphi},\cO_{\p1}(-2)\big)\longrightarrow\cO_{\p1}(-2)^{\oplus2}\longrightarrow0,
\end{align*}
because, being $\cE_{\quantum}$ locally free, we have $\sExt^1_{X}\big(\cE_{\quantum},\cO_{\p1}(-2)\big)=0$. Trivially the injective map on the left has a section. Thus taking the cohomology of the induced short exact sequence yields 
$$
\Hom_{X}\big(\cE_{\quantum,\varphi},\cO_{\p1}(-2)\big)=H^0\big(\sHom_{X}\big(\cE_{\quantum,\varphi},\cO_{\p1}(-2)\big)\big)=0,
$$
hence $\Ext^2_{X}\big(\cE_{\quantum,\varphi},\cE_{\quantum,\varphi}\big)=0$. 
Thus
$
\dim\Ext^1_{X}\big(\cE_{\quantum,\varphi},\cE_{\quantum,\varphi}\big)=1-\chi(\cE_{\quantum,\varphi},\cE_{\quantum,\varphi})$.
By applying $\Hom_{X}(\cE_{\quantum,\varphi},-\big)$, $\Hom_{X}(-,\cE_{\quantum}\big)$ and $\Hom_{X}(-,\cO_C\big)$ to Sequence \eqref{seqDeform} we deduce that
$$
\chi(\cE_{\quantum,\varphi},\cE_{\quantum,\varphi})=\chi(\cE_{\quantum},\cE_{\quantum})-\chi(\cE_{\quantum},\cO_C)-\chi(\cO_C,\cE_{\quantum})+\chi(\cO_C,\cO_C).
$$
We have $\chi(\cE_{\quantum},\cO_C)=2\chi(\cO_C)=2$ and $\chi(\cO_C,\cE_{\quantum})=-2\chi(\cO_C\otimes\cO_X(-h))=2$. Equality \eqref{RRGeneral} returns $\chi(\cE_{\quantum},\cE_{\quantum})=4-2\quantum$. From Sequence \eqref{seqLong} we obtain $\chi(\cO_C,\cO_C)=2\chi(\cO_X,\cO_C)-\chi(\cF,\cO_C)=0$, because $\cF\otimes\cO_C\cong\cN_{C\vert X}\cong\cO_C^{\oplus2}$. We deduce from the above computations that $\dim\Ext^1_{X}\big(\cE_\varphi,\cE_\varphi\big)=2\quantum+1$. 

The proof of the statement is then complete.
\end{proof}

Before starting the proof of Theorem \ref{tPrime} we collect below some helpful remarks. We use the same notation as in Construction \ref{conPrime}.

\begin{remark}
\label{rUnique}
If $Y\subseteq X$ is  any pure $1$--dimensional scheme, then
$$
\Ext^1_X\big(\cO_Y,\cE_{\quantum,\varphi}\big)=\left\lbrace\begin{array}{ll} 
0\quad&\text{if $C\ne Y$,}\\
1\quad&\text{if $C= Y$.}
\end{array}\right.
$$
Indeed, by applying $\Hom_X\big(\cO_Y,-\big)$ to Sequence \eqref{seqDeform} and taking into account Equality \eqref{SerreDual} we obtain 
$$
\Ext^1_X\big(\cO_Y,\cE_{\quantum,\varphi}\big)\cong \Hom_X\big(\cO_Y,\cO_C\big)\cong H^0\big( \sHom_X\big(\cO_Y,\cO_C\big)\big).
$$
We conclude by noticing that
$$
\sHom_X\big(\cO_Y,\cO_C\big)=\left\lbrace\begin{array}{ll} 
0\quad&\text{if $C\ne Y$,}\\
\cO_C\quad&\text{if $C=Y$.}
\end{array}\right.
$$
\end{remark}

\begin{remark}
\label{rBidual}
By applying $\Hom_X\big(-,\cE_\quantum\big)$ to Sequence \eqref{seqDeform} and taking into account Equality \eqref{SerreDual} we obtain $\Hom_X\big(\cE_{\quantum,\varphi},\cE_{\quantum}\big)\cong \bC$. 

By applying $\sHom_X\big(-,\cO_X\big)$ to Sequence \eqref{seqDeform} we deduce $\cE_{\quantum,\varphi}^{\vee\vee}\cong\cE_\quantum$, because $\cO_C$ is a torsion sheaf on $X$ and $\sExt^1_X\big(\cO_C,\cO_X\big)\cong\sExt^1_X\big(\cO_C,\omega_X\big)\otimes\cO_X(2h)=0$. 

The above isomorphisms show that the inclusion $\cE_{\quantum,\varphi}\to\cE_{\quantum}$ can be uniquely identified with the canonical monomorphism $\cE_{\quantum,\varphi}\to\cE_{\quantum,\varphi}^{\vee\vee}$, i.e. Sequence \eqref{seqDeform} is canonically isomorphic with 
\begin{equation*}
0\longrightarrow\cE_{\quantum,\varphi}\longrightarrow\cE_{\quantum,\varphi}^{\vee\vee}\longrightarrow\cO_C\longrightarrow0.
\end{equation*}
\end{remark}

Let $\cS_X\big(2;-\varepsilon h, \zeta, 0\big)\subseteq\cS_X$ be the locus of simple rank $2$ torsion--free sheaves with fixed Chern classes  $-\varepsilon h\in H^2(X)$, $\zeta\in H^4(X)$ and $0\in H^6(X)$ in what follows
In Proposition \ref{pKer} we proved that $\cE_{\quantum,\varphi}$ represents a point in $\cS_X\big(2;0,(\quantum+2)\ell,0\big)$. The core of the proof of Theorem \ref{tPrime} is the following proposition.

\begin{proposition}
\label{pDeform}
Let $X$ be an ordinary prime Fano threefold  with very ample $\cO_X(h)$. 

For each even $\quantum\ge2$, a general deformation $\cE\in \cS_X\big(2;0,(\quantum+2)\ell,0\big)$ of  $\cE_{\quantum,\varphi}$ is a $\mu$--stable vector bundle such that $\Ext^2_{X}\big(\cE,\cE\big)=0$ and $\cE\otimes\cO_Y\cong\cO_\p1^{\oplus2}$ for each general  $Y\in \Gamma(X)_0$.
\end{proposition}
\begin{proof}
In what follows we will denote by $\cS(2)$ the locus inside $\cS_X\big(2;0,2\ell,0\big)$ of bundles obtained via Construction \ref{conInstanton} from a general $Z\in\Gamma(X)_0$ and  by $\cS(\quantum+2)$  the component of $\cS_X\big(2;0,(\quantum+2)\ell,0\big)$ containing the point corresponding to the  sheaf $\cE_{\quantum,\varphi}$ for each even $\quantum\ge2$. Moreover, $\cS(\quantum+2)$ is smooth at $\cE_{\quantum,\varphi}$ of dimension
$$
\dim\Ext^1_{X}\big(\cE_{\quantum,\varphi},\cE_{\quantum,\varphi}\big)=2\quantum+1.
$$
Let $\cS^{bad}(\quantum+2)\subseteq \cS(\quantum+2)$ be the locus of sheaves obtained via Construction \ref{conPrime}. We claim that
\begin{equation}
\label{Bad}
\dim(\cS^{bad}(\quantum+2))\le 2\quantum.
\end{equation}

If $\quantum=2$, then the points in $\cS^{bad}(\quantum+2)$ are parameterized by two points in $\Gamma(X)_0$ (the conics $Z$ and $C$) and elements in $\bP\big(\Ext_X^1\big(\cI_{Z\vert X},\cI_{C\vert X}\big)\big)$ (the Sequence \eqref{seqBF}). Thanks to the computations in \cite[Proof of Step 3 of the proof of Theorem 4.1]{B--F1}, we know that the latter projective space reduces to a single point, whence we deduce that Inequality \eqref{Bad} is fulfilled.

If $\quantum\ge4$ is even, then the points in $\cS^{bad}(\quantum+2)$ are parameterized by the points of $\cS(\quantum)$ (the choice of $\cE_\quantum$), the points in $\Gamma(X)_0$ (the conic $C$) and elements in  $\bP\big(\Hom_{\p1}\big(\cO_{\p1}^{\oplus2},\cO_{\p1}\big)\big)$ (the morphism $\Hom_X\big(\cE_\quantum\otimes\cO_C,\cO_C\big)$ up to scalars). Thus Inequality \eqref{Bad} is again fulfilled.

It follows the existence of a flat family of torsion--free sheaves in $\cS(\quantum+2)$ over a smooth connected curve $\frak E\to S$ with $\frak E_{s_0}\cong\cE_{\quantum, \varphi}$ and $\frak E_s\not\in \cS^{bad}(\quantum+2)$ for $s\ne s_0$. Thanks to \cite[Satz 3]{B--P--S} we can assume $\Ext^2_{X}\big(\frak E_s,\frak E_s\big)=0$ for general $s\in S$.

For each $s\in S$ we have a natural exact sequence of sheaves over $X$
\begin{equation}
\label{seqBidual}
0\longrightarrow\frak E_s\longrightarrow(\frak E_s)^{\vee\vee}\longrightarrow\frak T_s\longrightarrow0.
\end{equation}
where $\frak T_s$ is a torsion sheaf on $X$ and the dual is taken with respect to $\cO_X$. Since $\frak E_s$ is torsion--free, it follows that the support $T_s$ of $\frak T_s$ has codimension at least $2$ in $X$ (see \cite[Corollary to Theorem II.1.1.8]{O--S--S}). Thanks to Proposition \ref{pKer}, by semicontinuity we can assume that $\frak E_s$ is simple and $\mu$--semistable, $h^0\big(\frak E_s\big)=h^1\big(\frak E_s(-h)\big)=0$ and  $\frak E_s\otimes\cO_Y\cong\cO_{\p1}^{\oplus2}$ for a general $Y\in\Gamma(X)_0$. Actually, we have two possible cases, namely either $T_s=\emptyset$ for general $s\in S$ or $T_s\ne\emptyset$ for each $s\in S$. 

In the former case we can find $s\in S$ such that $\frak E_s\cong(\frak E_s)^{\vee\vee}$ is reflexive. We deduce that $\frak E_s$ is a vector bundle, because $c_3(\frak E_s)=0$, thanks to \cite[Proposition 2.6]{Ha1}. We already know that $\Ext^2_{X}\big(\frak E_s,\frak E_s\big)=0$: an easy computation returns also the other dimensions. Moreover, $\cE\otimes\cO_Y\cong\cO_\p1^{\oplus2}$ for each general $Y\in \Gamma(X)_0$. Finally, $h^0\big(\frak E_s\big)=0$ by semicontinuity and Proposition \ref{pKer}, hence $\frak E_s$ can be assumed  $\mu$--stable by Lemma \ref{lHoppe} as well. Thus the statement is proved with $\cE:=\frak E_s$.

In what follows we will show that the case $T_s\ne\emptyset$ for each $s\in S$ does not occur. First we notice that if $h^0\big((\frak E_s)^{\vee\vee}(-h)\big)\ne0$, then we can find a subsheaf $\mathcal L\subseteq(\frak E_s)^{\vee\vee}$ with $\mu(\mathcal L)=h^3>0$. If $\cK:=\mathcal L\cap\frak E_s$, then $\mu(\cK)=\mu(\mathcal L)$, because the cokernel of the map $\cK\to\mathcal L$ induced by $\frak E_s\to(\frak E_s)^{\vee \vee}$ has support contained in $T_s$: this fact contradicts the $\mu$--semistability of $\frak E_s$, because $\cK\subseteq\frak E_s$. It follows from the cohomology of Sequence \eqref{seqBidual} tensored by $\cO_{X}(-h)$ that
\begin{equation}
\label{VanishingT}
h^0\big(\frak T_s(-h)\big)=0,
\end{equation}
hence $T_s$ cannot contain embedded points, i.e. it has pure dimension $1$. Equality \eqref{VanishingT} implies 
\begin{equation}
\label{C}
\chi( \frak T_s(th))=-h^1\big(\frak T_s(th)\big)
\end{equation}
for each $t\le-1$. Since $(\frak E_s)^{\vee \vee}$ is reflexive, it follows that $h^1\big((\frak E_s)^{\vee \vee}(th)\big)=0$ for $t\ll0$ thanks to \cite[Remark 2.5.1]{Ha1}. Thus the cohomology of Sequence \eqref{seqBidual} tensored by $\cO_X(th)$ yields
\begin{equation}
\label{A}
h^1\big(\frak T_s(th)\big)\le h^2\big(\frak E_s(th)\big)\le h^2\big(\cE_{\quantum,\varphi}(th)\big)
\end{equation}
when $t\ll0$ by semicontinuity.

We have $c_1((\frak E_s)^{\vee\vee})=c_1(\frak E_s)$. Moreover, we know that $c_3((\frak E_s)^{\vee\vee}(th))$, $c_3(\frak E_s(th))$ and $c_2(\frak T_s(th))$ are independent of $t$ for each $t\in\bZ$, thanks to \cite[Lemma 2.1]{Ha1}.  An easy Chern class computation from Sequence \eqref{seqBidual} then yields
\begin{equation*}
\begin{gathered}
c_1(\frak T_s(th))=0,\qquad hc_2(\frak T_s(th))=hc_2((\frak E_s)^{\vee\vee})-\quantum-2,\\ c_3(\frak T_s(th))=c_3((\frak E_s)^{\vee\vee})-2thc_2(\frak T_s),
\end{gathered}
\end{equation*}
because $c_i(\frak E_s)=c_i(\cE_{\quantum,\varphi})$. Equality \eqref{RRGeneral} for the sheaf $\frak T(th)$ yields
\begin{equation}
\label{RRT}
\chi(\frak T_s(th))=\frac{c_3((\frak E_s)^{\vee\vee})-hc_2(\frak T_s)}2-thc_2(\frak T_s).
\end{equation}

Since $\cE_\quantum$ is locally free, we know that $h^i\big(\cE_\quantum(th)\big)=0$ for $t\ll0$ and $i=1,2$, hence the cohomology of Sequence \eqref{seqDeform} tensored by $\cO_X(th)$ returns
\begin{equation}
\label{B}
h^2\big(\cE_{\quantum,\varphi}(th)\big)=h^1\big(\cO_C(th)\big)=-\chi(\cO_C(th))=-2t-1,
\end{equation}
for $t\ll0$. 

By combining Equalities \eqref{RRT} and \eqref{B} with Inequalities \eqref{A} and \eqref{C} we finally obtain
$$
\frac12\left( {c_3((\frak E_s)^{\vee\vee})} -hc_2(\frak T_s)\right)-thc_2(\frak T_s)\ge2t+1
$$
for $t\ll0$. Thus $hc_2(\frak T_s)\ge-2$: it follows that either the irreducible components of $T_s$ are up to two lines or $T_s$ is an integral conic.

In what follows we prove that the first configuration cannot occur. Let $L\subseteq T_s$ be a line: if equality holds, then there is an integer $a$ such that $\frak T_s\cong\cO_L(a)$. Assume that $L\ne T_s$, so that $\deg(T_s)=2$. Let $M\subseteq X$ be another line such that  $L+M=T_s$ in $A^2(X)$. The restriction map $\cO_{T_s}\to\cO_L$ induces an exact sequence
\begin{equation}
\label{seqExtension}
0\longrightarrow\cB\longrightarrow\frak T_s\longrightarrow\cA\longrightarrow0,
\end{equation}
where $\cA$ is supported on $L$ and $\cB$ on $M$. Since $\frak T_s$ is torsion free, then again there are integers $a,b$ such that $\cA\cong\cO_L(a)$ and $\cB\cong\cO_M(b)$.

We claim that for each line $N\subseteq X$ and integer $t$ we have
\begin{equation}
\label{Extension}
\Ext_X^1\big(\cO_N(t),\frak E_s\big)=0.
\end{equation}
Indeed, by semicontinuity it suffices to check that $\Ext_X^1\big(\cO_N(t),\cE_{\quantum,\varphi}\big)=0$. To show such a vanishing we apply the functor $\Hom_X\big(\cO_N(t),-\big)$ to Sequence \eqref{seqDeform} obtaining the exact sequence
$$
\Hom_X\big(\cO_N(t),\cO_C\big)\longrightarrow\Ext_X^1\big(\cO_N(t),\cE_{\quantum,\varphi}\big) \longrightarrow\Ext_X^1\big(\cO_N(t),\cE_{\quantum}\big).
$$
Notice that the space on the left vanishes because $C\in\Gamma(X)_0$ is general, hence  integral, thanks to Remark \ref{rUnique}. Equality \eqref{SerreDual} implies that
$$
\Ext_X^1\big(\cO_N(t),\cE_{\quantum}\big)\cong \Ext_X^2\big(\cE_{\quantum},\cO_N(t)\otimes\cO_X(-h)\big)\cong H^2\big(\cE_{\quantum}(-h)\otimes\cO_N(t)\big)=0,
$$
hence the claimed Vanishing \eqref{Extension} is completely proved.

By applying $\Hom_X\big(-,\frak E_s\big)$ to Sequence \eqref{seqExtension}, we deduce that $\Ext_X^1\big(\frak T_s,\frak E_s\big)=0$, thanks to Vanishing \eqref{Extension}, hence Sequence \eqref{seqBidual} should split. Thus  $\frak T_s$ should  be the torsion subsheaf of $(\frak E_s)^{\vee\vee}$, which is reflexive, hence torsion free, a contradiction. 

Therefore $T_s$ is an integral conic and  $\frak T_s\cong\cO_{T_s}(U)$ where $U\subseteq T_s\cong\p1$ is a divisor linearly equivalent to $up$ for some point $p\in T_s$. Equality \eqref{RRT} with $t=0$ and the Riemann--Roch theorem on $T_s$ yields $2u=c_3((\frak E)^{\vee\vee})$. On the one hand, the sheaf $(\frak E_s)^{\vee\vee}$ is reflexive, thus the locus where it is not locally free has degree $c_3((\frak E_s)^{\vee\vee})\ge0$, thanks to \cite[Proposition 2.6]{Ha1}, hence $u\ge0$. On the other hand, Equality \eqref{VanishingT} forces $u-2<0$. We conclude that $u\in\{\ 0,1\ \}$.

Let $u=1$: we will prove that such a case cannot occur, by showing that $\Ext_X^1\big(\cO_{T_s}(U),\frak E_s\big)=0$ as above. By semicontinuity it suffices to check that $\Ext_X^1\big(\cO_{T_s}(U),\cE_{\quantum,\varphi}\big)=0$. To this purpose we apply the functor $\Hom_X\big(\cO_{T_s}(U),-\big)$ to Sequence \eqref{seqDeform} obtaining the exact sequence
$$
\Hom_X\big(\cO_{T_s}(U),\cO_C\big)\longrightarrow\Ext_X^1\big(\cO_{T_s}(U),\cE_{\quantum,\varphi}\big) \longrightarrow\Ext_X^1\big(\cO_{T_s}(U),\cE_{\quantum}\big).
$$
Equality \eqref{SerreDual} implies that
$$
\Ext_X^1\big(\cO_{T_s}(U),\cE_{\quantum}\big)\cong \Ext_X^2\big(\cE_{\quantum},\cO_{T_s}(U)\otimes\cO_X(-h)\big)\cong H^2\big(\cE_{\quantum}(-h)\otimes\cO_{T_s}(U)\big)=0.
$$

We certainly have $\Hom_X\big(\cO_{T_s}(U),\cO_C\big)\cong H^0\big(\sHom_X\big(\cO_{T_s}(U),\cO_C\big)\big)$. The hypothesis $u=1$ and Remark \ref{rUnique} then imply $\Hom_X\big(\cO_{T_s}(U),\cO_C\big)=0$.

Thus the case $u=1$ is not possible, i.e. $u=0$, hence $\frak T_s\cong\cO_{T_s}$ for each $s\in S$. It follows the existence of a morphism $S\to \Gamma(X)_0$, such that $\frak T$ is exactly the pull--back of the universal conic on $\Gamma(X)_0$. 

Thus $\frak T\to S$ is a flat family and the flatness of the families $\frak E\to S$ and $\frak T\to S$ yields the flatness of the induced family $\bigcup_{s\in S}({\frak E}_s)^{\vee\vee}\to S$. In particular $(\frak E_{s_0})^{\vee\vee}\cong\cE_\quantum$ thanks to Remark \ref{rBidual}, hence $(\frak E_{s})^{\vee\vee}\in \cS(\quantum+2)$ and, consequently, $\frak E_s\in \cS^{bad}(\quantum+2)$ for general $s\in S$, contradicting our initial choice. We conclude that the case $T_s\ne\emptyset$ for each $s\in S$ cannot occur. 

It follows that for general $s\in S$ the sheaf $\cE:=\frak E_s$ is a vector bundle. Since the locus of $\mu$--stable sheaves inside $\cS_X(2;0,(\quantum+2)\ell,0)$ is open, it follows that we can assume that  $\cE$ is $\mu$--stable. By semicontinuity, thanks to Proposition \ref{pKer} and \cite{B--P--S}, we know that $\Ext^2_{X}\big(\cE,\cE\big)=0$ and $\cE\otimes\cO_Y\cong\cO_\p1^{\oplus2}$ for each general $Y\in \Gamma(X)_0$. 
\end{proof}

We are finally ready to prove Theorem \ref{tPrime} stated in the introduction by induction on even integers $\quantum\ge4$. 
The base step is  \cite[Theorem 4.1]{B--F1}.

\medbreak
\noindent{\it Proof of Theorem \ref{tPrime}.}
If $\quantum=2$, then Proposition \ref{pDeform} guarantees that the general deformation $\cE\in\cS_X(2;0,4\ell,0)$ of $\cE_{2,\varphi}$ is a vector bundle such that $\Ext^2_{X}\big(\cE,\cE\big)=0$. Moreover, by semicontinuity we can also assume $h^1\big(\cE(-h)\big)=0$ and $h^0\big(\cE\big)=0$. We deduce that $\cE$ is $\mu$--stable by Lemma \ref{lHoppe}, hence it is an instanton bundle with quantum number $hc_2(\cE)=hc_2(\cE_2)+hC=4$.

Let $\quantum\ge4$ be even and assume the existence of a $\mu$--stable instanton bundle $\cE_{\quantum}$ with quantum number $\quantum$ such that $\Ext^2_{X}\big(\cE_{\quantum},\cE_{\quantum}\big)=0$ and $\cE_{\quantum}\otimes\cO_Y\cong\cO_\p1^{\oplus2}$ for each general integral $Y\in \Gamma(X)_0$. The same argument used in the base step implies the existence of a $\mu$--stable instanton bundle $\cE\in\cS_X(2;0,(\quantum+2)\ell,0)$, such that $\Ext^2_{X}\big(\cE,\cE\big)=0$ and $\cE\otimes\cO_Y\cong\cO_\p1^{\oplus2}$ for each general  $Y\in \Gamma(X)_0$. 
\qed
\medbreak

We conclude this part by showing another possible approach to the construction of even instanton bundles on a Fano threefold $X$ with $i_X=1$.

\begin{remark}
\label{rCanonical}
Each Fano threefold $X$ with $i_X=\varrho_X=1$ and $g_X\ge4$ contains a canonical curve $Z\subseteq\p{g_X+1}$ as pointed out in \cite[Proposition 8]{Bea}. Indeed, if $\cE$ is the bundle constructed in \cite[Theorem 4.1]{B--F1}, then $\cE(h)$ is globally generated and its general section in $H^0\big(\cE(h)\big)$ vanishes on a canonical curve $Z$ with $p_a(Z)=g_X+2$. The same is true if $g_X=3$ and we restrict to general $X$.

More in general, let $X$ be any Fano threefold with $i_X=1$ and very ample $\cO_X(h)$. If $X$ contains a canonical curve $Z\subseteq\p{g_X+1}$, then $\omega_Z\cong\cO_X(h)\otimes\cO_Z$, $p_a(Z)=g_X+2$ and $\deg(Z)=2g_X+2$. Since $h^i\big(\cO_X(h)\big)=0$ for $i=1,2$, it follows from Theorem \ref{tSerre} the existence of vector bundle $\cE$ fitting into the exact sequence
\begin{equation}
\label{seqIndex1}
0\longrightarrow\cO_X(-h)\longrightarrow\cE\longrightarrow\cI_{Z\vert X}(h)\longrightarrow0.
\end{equation}
Notice  that $c_1(\cE)=0$ and $c_2(\cE)=Z-h^2$. Thus $hc_2(\cE)=4$.

If $\varrho_X=1$, then $\cE$ is a minimal $\mu$--stable even instanton bundle $\cE$. Indeed, the cohomology of Sequence \eqref{seqIndex1} returns $h^0\big(\cE\big)=0$, because $Z$, being a canonical curve, is non--degenerate inside $\p{g_X+1}$. Lemma \ref{lHoppe} then implies that $\cE$ is $\mu$--stable. 

The cohomology of Sequence \eqref{seqStandard} yields $h^1\big(\cI_{Z\vert X})\big)=0$, because $Z$ is  integral. Thus again the cohomology of Sequence \eqref{seqIndex1} tensored by $\cO_X(-h)$ returns $h^1\big(\cE(-h)\big)=0$. 

If $\varrho_X\ge2$, then the $\mu$--(semi)stability of $\cE$ cannot be deduced as easily as when $\varrho_X=1$. Indeed, if $\mu(\cO_X(D))\ge\mu(\cE)=0$ for a line bundle $\cO_X(D)\subseteq\cE$, then Sequence \eqref{seqIndex1} implies $h^0\big(\cI_{Z\vert X}(h-D)\big)=h^0\big(\cE(-D)\big)$. It follows that $\cE$ is $\mu$--stable if there are no surfaces $S\subseteq X$ through $Z$ with $\deg(S)\le 2g_X-2$. 

Notice that there always exists an infinite family of surfaces of degree $2g_X+1$ through $Z$, namely the cones projecting $Z$ from its points. For some results on the computation of the minimal degree of surfaces through a canonical curve see \cite{C--H}.
\end{remark}

\section{Some further properties of instanton bundles}
\label{sEarnest}
In this sections we first study the instanton bundles obtained in the previous section when restricted to suitable divisors. Then we will inspect their splitting behaviour when restricted to lines. 

The following property has been introduced in \cite{C--C--G--M}

\begin{definition}
Let $X$ be a Fano threefold.

We say that an instanton bundle $\cE$ on $X$ with $c_1(\cE)=-\varepsilon h$ is earnest if $h^1\big(\cE(-q_X^\varepsilon h-D)\big)=0$  for each smooth integral  divisor $D\subseteq X$.
\end{definition}

Let us spend few words on earnest instanton bundles. Let $D$ be an integral, smooth divisor on $X$ and assume that $\cO_X(h)$ is very ample. When $\varrho_X=1$, then $\cO_X(D)$ is very ample, hence the cohomology of Sequence \eqref{seqMaruyama} tensored by $\cE(-q_X^\varepsilon)$ implies $h^1\big(\cE(-q_X^\varepsilon h-D)\big)=h^0\big(\cE(-q_X^\varepsilon h)\otimes\cO_D\big)$. When $q_X^\varepsilon\ge1$ the dimension on the right vanishes by \cite[Theorem 3.1]{Ma1}. When $q_X^\varepsilon=0$, then $i_X=\varepsilon=1$: in this case $\mu(\cE\otimes\cO_D)<0$ and we can argue as in the previous case. 

When $\varrho_X\ge2$ such a deduction does not hold. Indeed there could be integral smooth divisors $D$ such that $\cO_X(D)$ is even not ample. E.g. such an issue always occurs for ordinary instanton bundles on the blow up of $\p3$ at a point or along a line: see respectively \cite{C--C--G--M, Cs--Ge, A--C--G}. Nevertheless, in these cases ordinary earnest instanton bundles $\cE$ are characterized by a finite number of further {\sl exotic instantonic conditions}, i.e. vanishings $h^1\big(\cE(-q_X^\varepsilon h-D)\big)=0$ for a suitable divisor $D\subseteq X$. 

In the case of the non--ordinary instanton bundles obtained in the previous sections we can prove the following result.

\begin{proposition}
\label{pEarnest}
Let $X$ be a Fano threefold  with very ample $\cO_X(h)$ and let $\cE$ be any bundle whose existence is guaranteed by Theorems \ref{tExistence} and \ref{tPrime}.
\begin{enumerate}
\item If $X\not\cong F_7$, then $\cE$ is earnest.
\item If $X\cong F_7$, then $\cE$ is earnest if and only if $\dim(Z\cap E)=0$. Moreover, if $\cE$ is not earnest, then $h^1\big(\cE(-h-D)\big)=0$ for each smooth integral divisor $D\ne E$ and $h^1\big(\cE(-h-E)\big)=1$.
\end{enumerate}
\end{proposition}
\begin{proof}
As pointed out above $\cE$ is automatically earnest if $\varrho_X=1$. Thus we have to prove the statement only when $X$ is one of the following threefolds: $F_{6,2}$, $F_{6,3}$, $F_7$.

The proof that $\cE$ is earnest when $X$ is either $F_{6,2}$ or $F_{6,3}$ is similar, hence we will deal with the first case in what follows. If $D\in \vert\alpha_1h_1+\alpha_2 h_2\vert\ne\emptyset$, then  $\alpha_1=Dh_2^2\ge0$ and $\alpha_2=Dh_1^2\ge0$. In particular $\cO_{F_{6,2}}(D)$ is globally generated, hence $\cO_{F_{6,2}}(h+D)$ is ample. It follows from \cite[Theorem 3.1]{Ma1} that $h^1\big(\cE(-h-D)\big)=0$ because $\cE$ is $\mu$--semistable (see the argument relating the earnest property with the $\mu$--semistability of the restriction of $\cE$ to any ample divisor). 

If $X\cong F_7$, then $\vert \alpha_1\xi+\alpha_2 f\vert$ contains a smooth $D$ if and only if either $\alpha_1=-\alpha_2=1$ or $\alpha_1,\alpha_2\ge0$ (see \cite[Remark 5.7]{C--C--G--M}). In the latter case $h+D$ is actually ample, hence $h^1\big(\cE(-h-D)\big)=0$ thanks to \cite[Theorem 3.1]{Ma1}. We deduce that $\cE$ is earnest if and only if $h^1\big(\cE(-h-E)\big)=0$. Since $h+E=2\xi$ and $2h+E=3\xi+f$, it follows from the cohomology of Sequences \eqref{seqInstanton} and \eqref{seqStandard} tensored by $\cO_{F_7}(-h-E)$ and \cite[Proposition 5.4]{C--C--G--M} that
$$
h^1\big(\cE(-h-E)\big)= h^1\big(\cI_{Z\vert F_7}(-h-E)\big)= h^0\big(\cO_{Z}(-h-E)\big).
$$
Notice that the dimension on the right is the sum of $h^0\big(\cO_{C}(-h-E)\big)$ as $C\cong\p1$ runs over the conics contained $Z$. If the class of $C$ in $A^2(F_7)$ is $\xi f$, then $(h+E)C=2$, hence $h^0\big(\cO_{C}(-h-E)\big)=0$. If the class of $C$ is $2(\xi-f)f$, then $(h+E)C=0$, hence $h^0\big(\cO_{C}(-h-E)\big)=1$. Since the components of $Z$ are pairwise disjoint, then at most one of them can be contained in $E\cong\p2$. We deduce that if $\dim(Z\cap E)=0$, then $h^0\big(\cO_{Z}(-h-E)\big)=0$, while if $\dim(Z\cap E)=1$, then $h^0\big(\cO_{Z}(-h-E)\big)=1$.
\end{proof}

We are interested in the splitting behaviour of instanton bundles when restricted to general lines. Indeed we finally prove also Theorem \ref{tTrivial} stated in the Introduction.

\medbreak
\noindent{\it Proof of Theorem \ref{tTrivial}.}
Assume that $\cE$ is obtained  via Construction \ref{conInstanton}.

Let us first consider the case $\varrho_X=1$. Since $\cE$ is $\mu$--stable, it follows that when $X\cong\p3$, i.e. $i_X=4$, then $\cE$ is generically trivial thanks to \cite[Corollary 1 of Theorem II.2.1.4]{O--S--S}. The same is true when $i_X=3$, thanks to \cite[Proposition 5.3]{C--F}. 

Now let $i_X=2$, $ \mathscr L\subseteq\Lambda(X)\times X$ the universal line and $\lambda\colon \mathscr L\to\Lambda(X)$  the projection to the first factor: the natural projection map $\chi\colon\mathscr L\to X$ is finite if $h^3\ge4$ (e.g. see \cite[Remark 2.2.7]{K--P--S}). If $h^3=3$ then it is generically finite and its fibres have positive dimension at  most in $30$ points (e.g. see \cite[p. 315 and Section 10]{Cl--Gr}). It follows that for each conic $C\subseteq X$, then $\lambda(\chi^{-1}(C))\subseteq\Lambda(X)$ is a union of a finite number of curves. Since $\Lambda(X)$ is a surface (e.g. see \cite[Propositions 3.5.6 and 3.5.8]{I--P}), it follows that the general line on $X$ does not intersect $C$. 

Thus the definition of $Z$ implies the existence of a line $L\cong \p1$ which does not intersect $Z$. It follows that the restriction of Sequence \eqref{seqInstanton} to $L$ induces a surjection $\cE\otimes\cO_L\twoheadrightarrow\cO_{\p1}$, hence $\cE\otimes\cO_L\cong\cO_{\p1}\oplus\cO_{\p1}(-1)$.

If $\varrho_X\ge2$ the argument is analogous and easier. Indeed, in this cases we know exactly the structure of $\Lambda(X)$ and $\Gamma(X)$ (see Remarks \ref{rFlag}, \ref{rSegre}, \ref{rBlow}), hence it is immediate to check that the general line does not intersect $Z$ in these cases.

Finally consider a bundle $\cE$ obtained as in Theorem \ref{tPrime}. Thanks to Lemma \ref{lConic}, we know that given general $Z,C\in \Gamma(X)_0$ the general line $L$ in each component of $\Lambda(X)$ does not intersect them. Restricting Sequence \eqref{seqInstanton} to $L$, we then deduce $\cE_2\otimes\cO_L\cong\cO_{\p1}^{\oplus2}$. Thus the restriction of Sequence \eqref{seqDeform} to the same line yields $\cE_{2,\varphi}\otimes\cO_L\cong\cO_{\p1}^{\oplus2}$ too. In particular every general deformation $\cE\in \cS_X\big(2;0, (\quantum+2)\ell,0\big)$  of $\cE_{2,\varphi}$ enjoys the same property by semicontinuity. 

If we now assume the existence of a generically trivial $\mu$--stable even instanton bundle $\cE_\quantum$ with quantum number $\quantum\ge4$, then the same argument as above yields the existence of a generically trivial $\mu$--stable even instanton bundle $\cE$ with quantum number $\quantum+2$, hence the statement is proved by induction on $\quantum$.
\qed
\medbreak

\section{Moduli spaces of non--ordinary instanton bundles}
\label{sModuli}
In this section we deal with the component of the moduli spaces of instanton bundles defined in the previous sections. 

Each instanton bundle $\cE$ obtained via Construction \ref{conInstanton} is $\mu$--stable, hence simple. Thus it represents a point in an open subset $\cS\cI_X\big(\varepsilon, \zeta)$ of $\cS_X\big(2; -\varepsilon h, \zeta,0\big)$ where $\zeta\in H^4(X)$ is such that $\zeta h=\quantum$.

\begin{proposition}
\label{pModuli}
Let $X$ be a Fano threefold with $i_X\ge2$ and very ample $\cO_X(h)$.

For each even integer $\quantum$ satisfying Inequality \eqref{BoundSharp} all the bundles obtained via Construction \ref{conInstanton} such that the class $\zeta\in H^4(X)$ of Z satisfies $\zeta h=\quantum$ represent smooth points in one and the same irreducible component $\cS\cI^0_X\big(\varepsilon, \zeta)\subseteq \cS\cI_X\big(\varepsilon, \zeta\big)$ with dimension 
$$
\dim(\cS\cI^0_X\big(\varepsilon, \zeta))=\left\lbrace\begin{array}{ll} 
8\quantum-5\quad&\text{if $i_X=4$,}\\
6\quantum-3\quad&\text{if $i_X=3$,}\\
4\quantum-h^3-3\quad&\text{if $i_X=2$.}
\end{array}\right.
$$
\end{proposition}
\begin{proof}
In order to prove the statement, it suffices to check that
\ref{conInstanton}
\begin{equation}
\label{dimModuli}
h^0\big(\cE\otimes\cE^\vee\big)=1,\qquad h^2\big(\cE\otimes\cE^\vee\big)=h^3\big(\cE\otimes\cE^\vee\big)=0,\\
\end{equation}
thanks to Equality \eqref{Ext12}.

To this purpose, since $\cE$ is simple, it follows that $h^0\big(\cE\otimes\cE^\vee\big)=1$ and, by Lemma \ref{lExt3}, also  $h^3\big(\cE\otimes\cE^\vee\big)=0$. It remains to check that $h^2\big(\cE\otimes\cE^\vee\big)=0$. We have
$$
h^i\big(\cE((1-q_X^\varepsilon)h)\big)=h^{3-i}\big(\cE(-q_X^\varepsilon h)\big).
$$
thanks to Equality \eqref{SerreDual}. In particular
\begin{gather*}
h^3\big(\cE((1-q_X^\varepsilon)h)\big)=h^0\big(\cE(-q_X^\varepsilon h)\big)\le h^0\big(\cE\big)=0,\\
h^2\big(\cE((1-q_X^\varepsilon)h)\big)=h^1\big(\cE(-q_X^\varepsilon h)\big)=0.
\end{gather*}
Thus the cohomology of Sequence \eqref{seqInstanton} tensored by $\cE^\vee\cong\cE(\varepsilon h)$ yields
$$
h^2\big(\cE\otimes\cE^\vee\big)=h^2\big(\cI_{Z\vert X}\otimes\cE((q_X^\varepsilon-1+\varepsilon)h)\big).
$$
The cohomology of Sequence \eqref{seqStandard} tensored by $\cE((q_X^\varepsilon-1+\varepsilon)h)$ implies
\begin{equation}
\label{BoundH2}
\begin{aligned}
h^2\big(\cI_{Z\vert X}\otimes\cE((q_X^\varepsilon-1+\varepsilon)h)\big)&\le h^2\big(\cE((q_X^\varepsilon-1+\varepsilon)h)\big)\\
&+h^1\big(\cO_Z\otimes\cE((q_X^\varepsilon-1+\varepsilon)h)\big).
\end{aligned}
\end{equation}

On the one hand, again the cohomologies of Sequences \eqref{seqInstanton} and \eqref{seqStandard} tensored by $\cE((q_X^\varepsilon-1+\varepsilon)h)$ and $\cO_{X}((i_X-1)h)$ respectively  imply
$$
h^2\big(\cE((q_X^\varepsilon-1+\varepsilon)h)\big)= h^2\big(\cI_{Z\vert X}((i_X-1)h)\big)\le h^1\big(\cO_{X}((i_X-1)h)\otimes\cO_Z\big).
$$
The latter dimension is zero because $Z$ is the disjoint union of rational curves. Thus the first righthand summand in Inequality \eqref{BoundH2} vanishes.

On the other hand,  for each integral conic $C\subseteq Z$, Equality \eqref{Normal} implies
$$
\cO_{C}\otimes\cE((q_X^\varepsilon-1+\varepsilon)h)\cong\cN_{C\vert X}\cong\cO_{\p1}(a_1)\oplus\cO_{\p1}(a_2)
$$
where $a_1\le a_2$ and $a_1+a_2=2i_X-2$ by adjunction on $X$. If $i_X\ge3$, then $C$ is the complete intersection of hypersurfaces of degrees $1$ and $i_X-2$, hence $(a_1,a_2)=(2,2i_X-4)$. Let $i_X=2$ and choose a general hyperplane section $H\subseteq X$ containing $C$. Thus $\cN_{H\vert X}\cong\cO_X(h)\otimes\cO_H$ and there exists a surjective morphism $\cN_{C\vert X}\to\cN_{H\vert X}$. It follows that $(a_1,a_2)$ is either $(0,2)$ or $(1,1)$. Since $Z$ is the disjoint union of conics in $X$, it follows that 
$$
h^1\big(\cO_Z\otimes\cE((q_X^\varepsilon-1+\varepsilon)h)\big)=h^1\big(\cN_{Z\vert X}\big)=0
$$
regardless of $i_X$. Thus the second righthand summand in Inequality \eqref{BoundH2} vanishes too and the proof of the Equalities  \eqref{dimModuli} is complete. 

The vanishing $h^2\big(\cE\otimes\cE^\vee\big)=0$ implies that $\cE$ in  $\cS_X\big(2; -\varepsilon h, \zeta,0\big)$ is smooth and lies in a component of dimension 
$\dim\Ext_X^1\big(\cE,\cE\big)\cong h^1\big(\cE\otimes\cE^\vee\big)$.
Moreover, $\cE$ is in the image of the natural rational map  $\Gamma(X)^{\times s}\dashrightarrow \cS_X\big(2; -\varepsilon h, \zeta,0\big)$. Thus all the bundles obtained via Construction \ref{conInstanton} lie in the same  component of $\cS\cI_X\big(\varepsilon, \quantum\big)$.
\end{proof}

When $\varrho_X=1$, each instanton bundle $\cE$ is automatically $\mu$--stable because $h^0\big(\cE\big)=0$. In \cite{M--M--PL, A--M} the authors proved the existence of ordinary instanton bundles which are not $\mu$--stable on $F_{6,2}$ and $F_{6,3}$. When $X\cong F_7$, each instanton bundle with quantum number $\quantum\le 14$ is $\mu$--stable (see \cite{C--C--G--M}).

The picture for non--ordinary instanton bundles is clearer.

\begin{proposition}
\label{pStable}
Let $X$ be a Fano threefold with $i_X\ge2$ and very ample $\cO_X(h)$.

Each non--ordinary instanton bundle on $X$ is $\mu$--stable if and only if $X\not\cong F_{6,2}$.
\end{proposition}
\begin{proof}
If $\varrho_X=1$ each instanton bundle is $\mu$--stable, hence we have to deal only with the case $\varrho_X\ge2$. To this purpose, thanks to Lemma \ref{lHoppe}, it suffices to look at the divisors $D$ such that $\mu(\cO_X(D))=\mu(\cE)$, when $X$ is $F_{6,2}$, $F_{6,3}$, $F_7$. 

Since $\mu(\cE)$ is either $-3$ or $-7/2$ according $X$ is either $F_{6,2}$ and $F_{6,3}$ or $F_7$ respectively, it follows that we can restrict to the two cases $F_{6,2}$ and $F_{6,3}$. In the latter case, using the notation of Remark \ref{rSegre}, we know that $h^2=2(h_2h_3+h_1h_3+h_1h_2)$, hence $\mu(\cO_X(D))$ is certainly even.

Thus, the only case we have to deal with is $X:=F_{6,2}$. In this case, we claim the existence of a strictly $\mu$--semistable non--ordinary instanton $\cE$. To this purpose, thanks to Lemma \ref{lHoppe}, it suffices to show that $h^0\big(\cE(-\alpha_1h_1-\alpha_2h_2)\big)=0$ for each $\cO_{F_{6,2}}(\alpha_1h_1+\alpha_2h_2)\subseteq\cE$ such that $3(\alpha_1+\alpha_2)=\mu(\cO_{F_{6,2}}(\alpha_1h_1+\alpha_2h_2))>\mu(\cE)=-3$, i.e. $\alpha_1+\alpha_2\ge0$, and that $h^0\big(\cE(h_1)\big)>0$.

In order to prove the claim, recall that $F_{6,2}$  can be also viewed as the incidence variety $\{\ (p,L)\ \vert\ p\in L\ \}\subseteq\p2\times(\p2)^\vee$. If we fix coordinates $u_0,u_1,u_2$ on $\p2$ and we denote by $v_0,v_1,v_2$ the dual coordinates in $(\p2)^\vee$, then the $F_{6,2}$ coincides with intersection of the image of the Segre map $\p2\times(\p2)^\vee\to\p8$ defined by
\begin{align*}
(x_0,\dots,x_8)=(u_0v_0,u_0v_1,u_0v_2,u_1v_0,u_1v_1,u_1v_2,u_2v_0,u_2v_1,u_2v_2)
\end{align*}
with the hyperplane $H:=\{\ x_0+x_4+x_8=0\ \}$ corresponding the incidence relation $p\in L$ translated analytically as $u_0v_0+u_1v_1+u_2v_2=0$. Thus the ideal of $F_{6,2}$ inside $H\cong\p7$ with coordinates $x_0,\dots,x_7$ is generated by the $2\times2$--minors of the matrix
$$
\left(\begin{array}{ccc}
x_0&x_1&x_2\\
x_3&x_4&x_5\\
x_6&x_7&-x_0-x_4
\end{array}\right).
$$

If $L\subseteq F_{6,2}$ is a line with class $h_1^2$, then it is the fibre of the projection map on the $u$--plane $\p2$. It follows that up to an automorphism of $F_{6,2}$ we can assume that $L:=\{\ x_0=\dots=x_5=0\ \}$. Thanks to \cite[Theorem 2.4]{N--N--S1} (or using any software for symbolic computation) one checks that the ideal
$$
I_C:=(x_0,\dots,x_5)^2+(x_2,x_5,x_0-x_4,x_1x_6-x_0x_7,x_4x_6-x_3x_7)
$$
defines a scheme $C\subseteq F_{6,2}$ with $\deg(C)=2$ and $p_a(C)=-2$.  The curve $C$ is obtained by doubling the line $L$ inside the  surface $S$ with ideal
$$
I_S:=(x_2,x_5,x_0-x_4,x_1x_6-x_0x_7,x_4x_6-x_3x_7).
$$

It is easy to check that $I_C+I_{F_{6,2}}=(x_1^2,x_3^2)+I_S+I_{F_{6,2}}$ and $x_4^2\in I_S+I_{F_{6,2}}$. It follows that the localization of $I_C+I_{F_{6,2}}$ at the prime ideals $(x_7)$ and $(x_6)$ coincides with the localizations of the ideals $(x_1^2)+I_S+I_{F_{6,2}}$ and $(x_3^2)+I_S+I_{F_{6,2}}$ respectively. Thus $C$ is locally complete intersection inside $F_{6,2}$ because $S$ is smooth at the points of $L$.

Moreover, each Cartier divisor on $C$ is completely identified by its degree which is twice the degree of the same divisor restricted to $L$: see \cite[Propositions 4.1 and 4.2]{B--E}. Since $\deg(\omega_{C})=-6$ and $\cO_{F_{6,2}}(-h_1-3h_2)\otimes\cO_L\cong\cO_{\p1}(-3)$, it follows that $\omega_{C}\cong \cO_{F_{6,2}}(-h_1-3h_2)\otimes\cO_{C}$. Thus
$$
\det(\cN_{C\vert X})\cong\omega_{C}\otimes\cO_{F_{6,2}}(2h)\cong \cO_{F_{6,2}}(h_1-h_2)\otimes\cO_{C},
$$
by adjunction. Since $\cO_{F_{6,2}}(h_1)$ is globally generated and $h_1^3=0$, it follows that
$$
\cO_{F_{6,2}}(-h_1-2h_2)\otimes\cO_C\cong\cO_{F_{6,2}}(-2h)\otimes\cO_C.
$$
Thus \cite[Theorem 2.1]{C--G--N} yields $h^0\big(\cO_{F_{6,2}}(-2h)\otimes\cO_C\big)=h^1\big(\cI_{C\vert\p7}(-2h)\big)=0$.

Let $Z$ be the union of $s:=(k-2)/2$ pairwise disjoint schemes of degree $2$ and genus $-2$ in $F_{6,2}$ supported on lines in class $h_1^2$ as above. The above discussion and Theorem \ref{tSerre} yield the existence of a vector bundle $\cE$ fitting into the exact sequence 
$$
0\longrightarrow\cO_{F_{6,2}}(-h_1)\longrightarrow\cE\longrightarrow\cI_{Z\vert F_{6,2}}(-h_2)\longrightarrow0.
$$
It is immediate to check that $c_1(\cE)=-h$, $hc_2(\cE)=\quantum$, $h^0\big(\cE\big)=0$ and $h^1\big(\cE(-h)\big)=h^1\big(\cI_{Z\vert F_{6,2}}(-h_1-2h_2)\big)$. Thus
$$
h^1\big(\cI_{Z\vert F_{6,2}}(-h_1-2h_2)\big)=h^0\big(\cO_{F_{6,2}}(-h_1-2h_2)\otimes\cO_Z\big)=sh^0\big(\cO_{F_{6,2}}(-2h)\otimes\cO_C\big)=0,
$$
whence we finally deduce $h^1\big(\cE(-h)\big)=0$. Moreover, $h^0\big(\cE(-\alpha_1h_1-\alpha_2h_2)\big)=0$ if $\alpha_1+\alpha_2\ge0$ and $h^0\big(\cE(h_1)\big)=1$. Thus the claim is proved and the proof of the proposition is complete.
\end{proof}

If $\quantum=2$, then  the bundle constructed in the proof of Proposition \ref{pStable} splits as $\cO_{F_{6,2}}(-h_1)\oplus\cO_{F_{6,2}}(-h_2)$: hence the above  construction extends Remark \ref{rSplit} to higher quantum numbers. Moreover, it is possible to show that $h^2\big(\cE\otimes\cE^\vee\big)=0$ regardless of the value of $\quantum$.

In order to conclude the analysis in the case $i_X\ge2$, we deal below with non--ordinary instanton bundles with low quantum numbers  and their moduli spaces. Recall that the minimal non--ordinary instanton bundles are the ones with quantum number $\quantum=2$ when $\varrho_X=1$ and $\quantum=4$ when $\varrho_X\ge2$ (see Theorems \ref{tSharp} and \ref{tExistence}). 

\begin{remark}
\label{rMinimalIndex1}
If $\varrho_X=1$, then each instanton bundle is automatically $\mu$--stable. The moduli space $\cS\cI_X\big(\varepsilon, \quantum\ell\big)$ is the open non--empty subset of bundles $\cE$ such that $h^1\big(\cE(-q_X^\varepsilon h)\big)=0$ inside the Maruyama moduli space $\cM_X(2;-\varepsilon h,\quantum\ell)$ of $\mu$--stable rank $2$ bundles (see \cite{Ma2} and \cite{Ma3}).

If $i_X=4,3,2$ and $\varepsilon=1,0,1$ respectively, then the space $\cM_X(2;-\varepsilon h,2\ell)$, hence $\cS\cI_X\big(\varepsilon, 2\big)$, is irreducible and smooth: moreover it is also rational (see \cite[Theorem 3.1]{H--S}, \cite[Theorem 2.1]{O--S} and \cite{S--W}  for the proof of this facts in the cases $i_X=4,3,2$ respectively). In particular $\cS\cI_X(\varepsilon, 2\ell\big)=\cS\cI_X^0(\varepsilon, 2\ell\big)$ in these cases and it actually coincides with $\cM_{X}(2;-h,2\ell)$ when $i_X\ne3$ (see \cite[Proposition 2.3]{H--S} and \cite[Proposition 2.1]{S--W}).

If $X=\p3$, then the moduli space $\cM_{\p3}(2;-h,4\ell)$ has been described in \cite[Theorems 1 and 2]{B--M}. More precisely (see \cite[Section 1.1]{B--M}) such a moduli space has two components, both rational. One of them is smooth of dimension $28$ and its points represent bundles $\cE$ such that $h^1\big(\cE(-2h)\big)=1$, thus it does not contain any point corresponding to an instanton bundle. The second component has dimension $27$  and its points represent bundles $\cE$ such that $h^1\big(\cE(-2h)\big)=0$: it follows that it coincides with $\cS\cI_{\p3}(1,4\ell)=\cS\cI_{\p3}^0(1,4\ell\big)$. More generally in \cite{Tik} the author defines a component $\cM\subseteq \cM_{\p3}(2;-h,\quantum\ell)$ for each even $\quantum$, shows that it is different from $\cS\cI_{\p3}^0(1,\quantum\ell\big)$ for $\quantum=6$ and conjectures that this is true for all $\quantum\ge6$.

If $X=Q$, then $\cM_Q(2;0,4\ell)$, hence $\cS\cI_Q\big(0,4\ell\big)$, is irreducible, reduced and unirational (see \cite[Theorem 3.4]{O--S}). It follows that $\cS\cI_Q(0,4\ell\big)=\cS\cI_Q^0(0,4\ell\big)$.

When $\varrho_X\ge2$, then $h^3\ge6$ and minimal non--ordinary instanton bundle must have quantum number $4$. We showed in Theorem \ref{tExistence} that such bundles exist. 

We have $h^2\big(\cE(h)\big)=0$ for such an $\cE$ by Lemma \ref{lH12}. Thus Equality \eqref{RRGeneral} implies  $h^0\big(\cE(h)\big)\ge h^3-3\ge3$. If $s\in H^0\big(\cE(h)\big)$ is a non--zero section whose zero--locus is a curve $Z$, then $\deg(Z)=hc_2(\cE(h))=hc_2(\cE)=4$. Equality \eqref{Normal} implies
$$
\det(\cN_{Z\vert X})\cong\det(\cE(h))\otimes\cO_Z\cong \cO_X(h)\otimes\cO_Z,
$$
hence the adjunction formula yields
$$
\omega_Z\cong\omega_X\otimes\cO_X(h)\otimes\cO_Z\cong\cO_X(-h)\otimes\cO_Z.
$$
Thus $p_a(Z)=-1$.
\end{remark}

We conclude the analysis of the case $i_X\ge2$ by characterizing minimal instanton bundles on $F_{6,2}$ which are not $\mu$--semistable as the ones associated to double structures $Z$ with $p_a(Z)=-2$ on lines inside $ F_{6,2}$.

\begin{remark}
In what follows we will use the notation of Remark \ref{rFlag}. It is easy to see that the condition $\mu(\cO_{F_{6,2}}(D))=\mu(\cE)$ forces the existence of an integer $\alpha$ such that $D=\alpha h_1-(1+\alpha)h_2$. Let $s\in H^0\big(\cO_{F_{6,2}}(-D)\big)$ be a non--zero section. If the zero--locus of $s$ contains a surface $S$, then
$$
\deg(S)=Sh^2=\mu(\cO_{F_{6,2}}(S))\le \mu(\cE(-D))=\mu(\cE)-\mu(\cO_{F_{6,2}}(D))=0.
$$
We deduce that $(s)_0=Z$ is a curve with degree $c_2(\cE(-D))h$. One easily checks that
$c_2(\cE(-D))h=2(1-\alpha-\alpha^2)$ which is negative, unless $\alpha\in\{\ -1,0\ \}$, i.e. $D=h_i$. Thus we deduce $\deg(Z)=2$. 

Let $\beta_1h_2^2+\beta_2 h_1^2$ be the class of $Z$. Since $\beta_i=Zh_i\ge0$ and $\beta_1+\beta_2=\deg(Z)=2$, it follows that the possible values of $(\beta_1,\beta_2)$ are $(2,0)$, $(1,1)$ and $(0,2)$. Let $D=h_1$: Equality \eqref{Normal} and the adjunction formula imply $2p_a(Z)-2=(-h_1-3h_2)Z$. Thus $p_a(Z)$ is $-2$, $-1$ and $0$ respectively in the three cases.

If the class of $Z$ is $2h_2^2$, then $p_a(Z)=0$, hence $Z$ would be a conic supported on a line. We can exclude this case as in Remark \ref{rFlag}.

If the class of $Z$ is $h_1^2+h_2^2$, then $p_a(Z)=-1$. In this case $Z$ is either the union of two disjoint lines, or a double structure on a line. The latter case cannot occur because $h_1^2+h_2^2$ is not divisible in $A^2(F_{6,2})$, hence $Z$ is the union of lines $L_i$ with class is $h_i^2\in A^2(F_{6,2})$. Equality \eqref{Normal} implies that
$$
\det(\cN_{L_i\vert F_{6,2}})\cong\det(\cE(h_1))\otimes\cO_{L_i}\cong\det(\cO_{F_{6,2}}(h_1-h_2))\otimes\cO_{L_i}
$$
because $L_1\cap L_2=\emptyset$. Thus this case cannot occur too, because $(h_1-h_2)h_1^2=-1$ and $\cN_{L_1\vert F_{6,2}}\cong\cO_{\p1}^{\oplus2}$ as shown above.

It follows that the class of $Z$ is $2h_1^2$ and $p_a(Z)=-2$. Thus $Z$ is necessarily a double structure on a line $L$ with class $h_1^2$ and Sequence \eqref{seqSerre} becomes
$$
0\longrightarrow\cO_{F_{6,2}}\longrightarrow\cE(h_1)\longrightarrow\cI_{Z\vert F_{6,2}}(h_1-h_2)\longrightarrow0,
$$
hence $h^0\big(\cE(h_1)\big)=1$, i.e. the curve $Z$ completely identifies $\cE$. 

Tensoring the above sequence by $\cO_{F_{6,2}}(h_2-h_1)$, one deduces $h^0\big(\cE(h_2)\big)=0$. Thus if we make a similar construction with $h_2$ instead of $h_1$, we obtain a different family of $\mu$--semistable minimal instanton bundles.
\end{remark}

The instanton bundles obtained in Theorem \ref{tPrime} are $\mu$--stable, hence simple. Thus they represent points in an open subset $\cS\cI_X\big(\varepsilon, \quantum\ell)\subseteq\cS_X\big(2; -\varepsilon h, \quantum \ell,0\big)$.

Theorem \ref{tPrime} is less explicit than Construction \ref{conInstanton}. Thus we have no way to confront the bundles constructed in Section \ref{sPrime} in order to show that they lie in the same component. Nevertheless, at least the following statement holds.

\begin{proposition}
\label{pModuliPrime}
Let $X$ be an ordinary prime Fano threefold  with very ample $\cO_X(h)$.

For each even integer $\quantum$ satisfying Inequality \eqref{BoundSharp} there exists an earnest instanton bundle $\cE$ with quantum number $\quantum$ representing a smooth point in an irreducible component $\cS\cI^0_X\big(\varepsilon, \quantum\ell)\subseteq \cS\cI_X\big(\varepsilon, \quantum\ell\big)$ with dimension 
$$
\dim(\cS\cI^0_X\big(\varepsilon, \quantum\ell))=2\quantum-3.
$$
\end{proposition}
\begin{proof}
As pointed out in the proof of Theorem \ref{tPrime} we can construct a $\mu$--stable instanton bundle $\cE$ with $\Ext^2_{X}\big(\cE,\cE\big)=0$ and quantum number $\quantum$. Thanks to Equality \eqref{Ext12} we obtain $\dim\Ext^1_{X}\big(\cE,\cE\big)=2\quantum-3$.
\end{proof}

\section{Monads for non--ordinary instanton bundles when $i_X\ge3$}
\label{sMonad}
In this section we prove the existence of monads for the bundles we are interested in when $X$ is either $\p3$ or $Q$.

\begin{proposition}
\label{pMonad}
Each odd instanton bundle $\cE$ on $\p3$ with quantum number $\quantum$ is the cohomology of a monad $\cC^\bullet$ of the form
\begin{equation}
\label{Monad}
0\longrightarrow (H\otimes\cO_{\p3}(h))^\vee\mapright{u\psi^\vee(-h)}(K\otimes\Omega_{\p3}(2h))^\vee\oplus K\otimes\Omega_{\p3}(h)\mapright{\psi}H\otimes\cO_{\p3}\longrightarrow0,
\end{equation}
where $H$ and $K$ are vector spaces with
$$
\dim(H)=\frac{3\quantum}2-1,\qquad \dim(K)=\frac{\quantum}2,
$$
and
$$
u\colon (K\otimes\Omega_{\p3}(2h)\oplus (K\otimes\Omega_{\p3}(h))^\vee)(-h)\to (K\otimes\Omega_{\p3}(2h))^\vee\oplus K\otimes\Omega_{\p3}(h)
$$
is skew--symmetric.

Conversely, if the cohomology $\cE$ of the monad $\cC^\bullet$ is a vector bundle, then $\cE$ is an odd instanton bundle with $c_2(\cE)=\quantum$ on $\p3$.
\end{proposition}
\begin{proof}
If $\cE$ is an odd instanton bundle on $\p3$, then it  is the cohomology in degree $0$ of a complex ${\cC}^\bullet$ with $i^{th}$--module
$$
{\cC}^i:=\bigoplus_{p+q=i}H^q\big(\cE(ph)\big)\otimes \Omega^{-p}_{\p3}(-ph)
$$
where, as usual, $\Omega^t_{\p3}:=\wedge^t \Omega_{\p3}$, thanks to \cite[Beilinson's Theorem (Strong form)]{A--O1}. Our first task is to prove that $h^{q}\big(\cE(ph)\big)$ is the number in position $(p,q)$ in table 2.
\begin{table}[H]
\label{TableEpq}
\centering
\bgroup
\def\arraystretch{1.5}
\begin{tabular}{cccccc}
\cline{2-5}
\multicolumn{1}{c}{\phantom{000000}} &\multicolumn{1}{|c|}{0} & \multicolumn{1}{c|}{0} & \multicolumn{1}{c|}{0} & \multicolumn{1}{c|}{0} & $q=3$ \\ 
\cline{2-5}
\multicolumn{1}{c}{} &\multicolumn{1}{|c|}{$\frac{3\quantum}2-1$} & \multicolumn{1}{c|}{$\frac{\quantum}2$} & \multicolumn{1}{c|}{0} & \multicolumn{1}{c|}{0} & $q=2$ \\ 
\cline{2-5}
\multicolumn{1}{c}{} &\multicolumn{1}{|c|}{0} & \multicolumn{1}{c|}{0} & \multicolumn{1}{c|}{$\frac{\quantum}2$} & \multicolumn{1}{c|}{$\frac{3\quantum}2-1$} & $q=1$ \\ 
\cline{2-5}
\multicolumn{1}{c}{} &\multicolumn{1}{|c|}{0} & \multicolumn{1}{c|}{0} & \multicolumn{1}{c|}{0} & \multicolumn{1}{c|}{0} & $q=0$ \\ 
\cline{2-5}
&$p=-3$ & $p=-2$ & $p=-1$ & $p=0$
\end{tabular}
\egroup
\caption{The values of $h^{q}\big(\cE(ph)\big)$ for $\p3$}
\end{table}

Indeed, thanks to Proposition \ref{pNatural} the cohomology of $\cE$ is natural in degree $0$, hence  $h^{q}\big(\cE\big)=0$ when $q\ne1$. Moreover, $h^{2}\big(\cE(-h)\big)=h^1\big(\cE(-2h)\big)=0$ by definition. The vanishing $h^0\big(\cE\big)=0$ implies $h^0\big(\cE(ph)\big)=0$ for $-3\le p\le 0$. Equality \eqref{SerreDual} then returns $h^3\big(\cE(-ph)\big)=0$ in the same range. For the same reason $h^q\big(\cE(ph)\big)=h^{3-q}\big(\cE((-p-3)h)\big)$ for $p=0,-1$ and $q=1,2$. In order to compute $h^1\big(\cE\big)=h^2\big(\cE(-3h)\big)$ and $h^1\big(\cE(-h)\big)=h^2\big(\cE(-2h)\big)$ we apply Equality \eqref{RRGeneral} to $\cE$ and $\cE(-h)$ respectively. 

Let $H:=H^1\big(\cE\big)$ and $K:=H^1\big(\cE(-h)\big)$. Equality \eqref{SerreDual} implies
$$
H^2\big(\cE(-3h)\big)\cong H^1\big(\cE\big)^\vee,\qquad H^2\big(\cE(-2h)\big)\cong H^1\big(\cE(-h)\big)^\vee,
$$
hence the existence of a monad $\cC^\bullet$ of the form
$$
0\longrightarrow H^\vee\otimes\cO_{\p3}(-h)\mapright{\eta}(K\otimes\Omega_{\p3}(2h))^\vee\oplus K\otimes\Omega_{\p3}(h)\mapright{\psi}H\otimes\cO_{\p3}\longrightarrow0,
$$
follows directly from the isomorphism $\Omega^t_{\p3}(4h)\cong(\Omega^{3-t}_{\p3})^\vee$ and Table 2 above.

The monads $\cC^\bullet$ and $(\cC^\bullet)^\vee$ satisfy the hypothesis of \cite[Lemma II.4.1.3]{O--S--S}, hence the skew--symmetric isomorphism $\cE^\vee(-h)\cong\cE$ induces an isomorphism of the monads $(\cC^\bullet)^\vee(-h)\cong\cC^\bullet$ such that the map $u\colon\cC_0^\vee(-h)\to\cC_0$ is skew--symmetric. The same argument described after the proof of \cite[Corollary 2 of Lemma II.4.1.3]{O--S--S} finally yields Monad \eqref{Monad}.

Conversely, if the cohomology $\cE$ of Monad \eqref{Monad} is a vector bundle, then it is easy to check that $\rk(\cE)=2$, 
$c_1(\cE)=-1$, $ c_2(\cE)=\quantum$. We have to check that $h^0\big(\cE\big)=h^1\big(\cE(-2h)\big)=0$. To this purpose consider the two short exact sequences
\begin{equation}
\label{Display1}
\begin{gathered}
0\longrightarrow \cK\longrightarrow (K\otimes\Omega_{\p3}(2h))^\vee\oplus K\otimes\Omega_{\p3}(h)\longrightarrow H^\vee\otimes\cO_{\p3}\longrightarrow0,\\
0\longrightarrow H\otimes\cO_{\p3}(-h)\longrightarrow \cK\longrightarrow\cE\longrightarrow0.
\end{gathered}
\end{equation}

The cohomology of Sequences \eqref{Display1} easily yields
\begin{gather*}
2h^0\big(\cE\big)\le \quantum h^0\big(\Omega^\vee_{\p3}(-2h)\big)+\quantum h^0\big(\Omega_{\p3}(h)\big),\\
2h^1\big(\cE(-2h)\big)\le \quantum h^1\big(\Omega^\vee_{\p3}(-4h)\big)+\quantum h^1\big(\Omega_{\p3}(-h)\big).
\end{gather*}
The statement then follows by computing the cohomology of suitable twists of the Euler exact sequence
and its dual.
\end{proof}

\begin{remark}
As pointed out in \cite[Section 5.1]{El--Gr}, if $\cE$ is an  odd instanton bundle on $\p3$ with quantum number $\quantum$, then it is the cohomology of a monad of the form
\begin{align*}
0\longrightarrow\cO_{\p3}(-2h)^{\oplus \quantum}\oplus\cO_{\p3}(-h)^{\oplus t_0}&\longrightarrow\cO_{\p3}(-h)^{\oplus \quantum+t_0+1}\oplus\cO_{\p3}^{\oplus \quantum+t_0+1}\\
&\longrightarrow\cO_{\p3}^{\oplus t_0}\oplus\cO_{\p3}(h)^{\oplus \quantum}\longrightarrow 0
\end{align*}
where $t_0$ is the number of minimal generators of $\bigoplus_{t\in\bZ}H^1\big(\cO_{\p3}(t)\big)$ lying in $H^1\big(\cO_{\p3}\big)$.
\end{remark}

We now turn the attention to even instanton bundles on $Q$. If $\cE$ is even, then it is $\mu$--stable. Thus, we have the following well--known result.

\begin{proposition}
\label{pMonadQuadric}
Each even instanton bundle $\cE$ on $Q$ with $c_2(\cE)=\quantum \ell$ is the cohomology of a monad $\cC^\bullet$ of the form
\begin{equation}
\label{MonadQuadric}
0\longrightarrow H^\vee\otimes\cO_{Q}(-h)\mapright{u\psi^\vee} K\otimes\cO_{Q}\mapright{\psi}H\otimes\cO_{Q}(h)\longrightarrow0,
\end{equation}
where $H$ and $K$ are vector spaces with
$$
\dim(H)=\frac{\quantum}2,\qquad \dim(K)=\quantum+2
$$
and 
$$
u\colon K\otimes\cO_{Q}\to K^\vee\otimes\cO_{Q}
$$
is skew--symmetric.

Conversely, if the cohomology $\cE$ of the monad $\cC^\bullet$ is a vector bundle, then $\cE$ is an even instanton bundle with $c_2(\cE)=\quantum \ell$ on $Q$.
\end{proposition}
\begin{proof}
For the existence of a monad see \cite[Proposition 1.1]{O--S}. We can deduce its self--duality as in the proof of Theorem \ref{pMonad}. In the same way we prove the converse part of the statement.
\end{proof}

\begin{remark}
\label{rConfront}
As pointed out in \cite[Theorem 3.5]{C--MR} a rank $2$ bundle $\cE$ on $Q$ is the cohomology of Monad \eqref{MonadQuadric} if and only if there is at most one $i\ge0$ such that $h^i\big(\cE(-th)\big)\ne0$ where $-2\le t\le0$, at most one $i\ge0$ such that $h^i\big(\cE\otimes\cS(-2h)\big)\ne0$ and $h^i\big(\cE\otimes\cS(-h)\big)=0$ for all $i\ge0$.

Thus Theorem \ref{pMonadQuadric} guarantees that all such sets of vanishings are automatically satisfied by an even instanton bundle on $Q$. Notice that the first set of vanishings above is exactly Proposition \ref{pNatural}.
\end{remark}

Thanks to  \cite[Corollary 2 of Theorem II.2.1.4]{O--S--S} and \cite[Proposition 5.3]{C--F} every instanton bundle $\cE$ on either $\p3$ or $Q$ is always generically trivial. We conclude the section by dealing with the locus of lines where the splitting of $\cE$ is not balanced.

\begin{definition}
\label{dJumping}
Let $X$ be a Fano threefold  with very ample $\cO_X(h)$.

If $\cE$ is an instanton bundle on $X$ with $c_1(\cE)=-\varepsilon h$, then a line $L\subseteq X$ is called jumping line of $\cE$ if $h^1\big(\cE((i_X+\varepsilon-2q_X^\varepsilon-1)h)\otimes\cO_L\big)\ne0$.  The locus of jumping lines of $\cE$ inside $\Lambda(X)$ will be denoted by $J_\cE$.

We say that $\cE$ has expected splitting type if the codimension of all the components of $J_\cE$ inside $\Lambda(X)$ is $1+\varepsilon$.
\end{definition}

\begin{example}
\label{eHulek}
If $i_X\ge3$ and $\cE$ is an even instanton bundle, then \cite[Corollary 2 of Theorem II.2.1.4]{O--S--S} and \cite[Proposition 5.3]{C--F} imply that $\cE$ has expected splitting type. This is no longer true if $i_X\le2$. Indeed in this case the projection $\chi\colon\mathscr L\to X$ can have disconnected fibres as pointed out in \cite[Section 3.7]{Kuz}, hence the standard approach of \cite{O--S--S,C--F} does not work.

When $X=\p3$ and  $\cE$ is odd with quantum number $2$, then $\cE$ has expected splitting type, thanks to \cite[Section 4]{H--S}. As far as we know, there are no other general results even when $i_X\ge3$: indeed one cannot say anything even about the analogous problem for bundles on $\p2$: see \cite[Example 10.7.2]{Hu}.
\end{example}

Let $\widehat{\lambda}$ and $\widehat{\chi}$ the natural projections $\Lambda(X)\times X$ on the first and second factor respectively. We will denote by $\lambda$ and $\chi$ their restrictions to the universal line $\mathscr L\subseteq\Lambda(X)\times X$. The following lemma follows immediately by the above definition.

\begin{lemma}
\label{lJumpingLocus}
Let $X$ be a Fano threefold with very ample $\cO_X(h)$.

If $\cE$ is an instanton bundle on $X$ with $c_1(\cE)=-\varepsilon h$, then the locus $J_\cE\subseteq\Lambda(X)$ is the support of the sheaf $R^1\lambda_*\chi^*\cE((i_X+\varepsilon-2q_X^\varepsilon-1)h)$.
\end{lemma}
\begin{proof}
The open locus $\Lambda(X)\setminus J_\cE$ is by definition \ref{dJumping} the set of $L\in \Lambda(X)$ such that $h^1\big(\cE((i_X+\varepsilon-2q_X^\varepsilon-1)h)\otimes\cO_L\big)=0$. Thanks to \cite[Corollary III.12.9]{Ha2}, it is thus the locus where $R^1\lambda_*\chi^*\cE((i_X+\varepsilon-2q_X^\varepsilon-1)h))=0$. 
\end{proof}

We first deal with instanton bundles on $\p3$. Recall that $\Lambda(\p3)=G(2,4)$ is the Grassmannian of lines in $\p3$ which is a quadric in $\p5$ endowed with the induced natural polarization $\cO_{\Lambda(\p3)}(1)$. In this case $J_\cE$ is the support of the sheaf $R^1\lambda_*\chi^*\cE$. 

The quadric $\Lambda(\p3)$ contains two families of planes which are classically known as $\alpha$--planes (the lines through a fixed point) and $\beta$--planes (the lines contained in a fixed plane). Both these families of planes correspond to the non--zero sections of the two rank $2$ spinor bundles on $\Lambda(\p3)$ (see \cite{Ott2}).

As pointed out in  \cite[Section 1.5]{G--H}, $H^*(\Lambda(\p3))$ is generated by the Schubert cycles. More precisely, $H^2(\Lambda(\p3))$ is generated by $\sigma_1$ (representing the set of lines in $\Lambda(\p3)$ intersecting a fixed one),  $A^2(\Lambda(\p3))$  by $\sigma_2$ and $\sigma_{1,1}$ (representing respectively the classes of the $\alpha$--planes and $\beta$--planes in $\Lambda(\p3)$). Moreover, the relation $\sigma_1^2=\sigma_2+\sigma_{1,1}$ holds inside $\Lambda(\p3)$.

Let $\cS$ be the spinor bundle corresponding to the $\alpha$--planes. 

\begin{proposition}
Let $\cE$ be an odd instanton bundle on $\p3$ with quantum number $\quantum$. 

If $\cE$ has expected splitting type, then $J_{\cE}\subseteq\Lambda(\p3)$ is endowed with a sheaf $\mathscr G$ fitting into
\begin{equation}
\label{seqJumping}
\begin{aligned}
0\longrightarrow \cO_{\Lambda(\p3)}(-\quantum)&\longrightarrow (K\otimes\cO_{\Lambda(\p3)}(1))^\vee\oplus (K\otimes\cS)\\
&\mapright\Psi H\otimes\cO_{\Lambda(\p3)}\longrightarrow \mathscr G\longrightarrow 0.
\end{aligned}
\end{equation}
Moreover, the class of $J_\cE$ in $H^*(\Lambda(\p3))$ is
$$
\frac{k^2}{2}\sigma_2+\frac{k(k-1)}{2}\sigma_{1,1}.
$$
\end{proposition}
\begin{proof}
Splitting Monad \eqref{Monad} we obtain the exact sequences
\begin{equation}
\label{seqDisplay}
\begin{gathered}
{\begin{aligned}
0\longrightarrow \chi^*\cK\longrightarrow (K\otimes\chi^*\Omega_{\p3}(2h))^\vee&\oplus (K\otimes\chi^*\Omega_{\p3}(h))\\
&\longrightarrow H\otimes\chi^*\cO_{\p3}\longrightarrow 0,
\end{aligned}}\\
0\longrightarrow H^\vee\otimes\chi^*\cO_{\p3}(-h)\longrightarrow\chi^*\cK\longrightarrow \chi^*\cE\longrightarrow 0.
\end{gathered}
\end{equation}
As in the previous proof we apply the functor $\lambda_*$ to the two above sequences. In particular we have to identify for all $i\ge0$ the sheaves $R^i\lambda_*\chi^*\Omega^\vee_{\p3}(-2h)$,  $R^i\lambda_*\chi^*\Omega_{\p3}(h)$, $R^i\lambda_*\chi^*\cO_{\p3}(-th)$ where $t=0,1$.

The structure sheaf of the universal line fits into an exact sequence of the form
\begin{equation}
\begin{aligned}
\label{seqUniversal}
0\longrightarrow \cO_{\Lambda(\p3)}(-1)\boxtimes\cO_{\p3}(-2h)&\longrightarrow   \cS\boxtimes\cO_{\p3}(-h) \\
&\longrightarrow \cO_{\Lambda(\p3)\times {\p3}}\longrightarrow \cO_{\mathscr L}\longrightarrow0.
\end{aligned}
\end{equation}
We will compute $R^i\lambda_*\chi^*\Omega^\vee_{\p3}(-2h)$,  $R^i\lambda_*\chi^*\Omega_{\p3}(h)$, $R^i\lambda_*\chi^*\cO_{\p3}(-th)$ where $t=0,1$ by applying the functor $\widehat{\lambda}_*$ to Sequence \eqref{seqUniversal} tensored by $\widehat{\chi}^*\Omega^\vee_{\p3}(-2h)$,  $\widehat{\chi}^*\Omega_{\p3}(h)$, $\widehat{\chi}^*\cO_{\p3}(-th)$ where $t=0,1$, making use of the projection formula \cite[Exercise III.8.3]{Ha2} and taking into account that $\widehat{\lambda}_*\widehat\cF\cong{\lambda}_*\cF$ for each sheaf supported on $\mathscr L$. 

Let $\cG$ be a sheaf on $\p3$. For each $L\in\Lambda(\p3)$ we denote by $k(L)$ the residue field of $\cO_{\Lambda(\p3)}$ at $L$. Trivially $h^i\big(\widehat{\chi}^*\cG\otimes k(L)\big)=h^i\big(\cG\big)$, hence the function $L\mapsto h^i\big(\widehat{\chi}^*\cG\otimes k(L)\big)$ is constant on $\Lambda(\p3)$. Thus \cite[Corollary III.12.9]{Ha2} implies that
$$
R^i\widehat{\lambda}_*(\cO_{\Lambda(\p3)}(-s)\boxtimes\cG)\cong \cO_{\Lambda(\p3)}(-s)\otimes R^i\widehat{\lambda}_*\widehat{\chi}^*\cG\cong\cO_{\Lambda(\p3)}(-s)\otimes H^i\big(\cG\big)
$$
In this way we deduce that 
$$
R^0{\lambda}_*{\chi}^*\cO_{\p3}\cong\cO_{\Lambda(\p3)},\qquad R^0{\lambda}_*{\chi}^*\Omega_{\p3}^\vee(-2h)\cong\cO_{\Lambda(\p3)}(-1),\qquad R^0{\lambda}_*{\chi}^*\Omega_{\p3}(h)\cong\cS,
$$
all the other $R^i$'s being zero. Replacing the above isomorphisms in the first of Sequence \eqref{seqDisplay} and taking into account that the second of such sequences gives $R^i\lambda_*\chi^*\cK\cong R^i\lambda_*\chi^*\cE$ we finally obtain the exact sequence
\begin{equation}
\label{seqJumpingFirst}
\begin{aligned}
0\longrightarrow \lambda_*\chi^*\cE&\longrightarrow (K\otimes\cO_{\Lambda(\p3)}(1))^\vee\oplus (K\otimes\cS)\\
&\mapright\Psi H\otimes\cO_{\Lambda(\p3)}\longrightarrow R^1\lambda_*\chi^*\cE\longrightarrow0,
\end{aligned}
\end{equation}
where the map $\Psi$ is induced by $\psi^\vee(-h)$. In particular $R^1\lambda_*\chi^*\cE$ has rank $1$. Since $\cE$ has expected splitting type, then  all the components of $J_\cE$ have codimension $2$. In this case $c_1(R^1\lambda_*\chi^*\cE)=0$, hence computing it from Sequence \eqref{seqJumpingFirst} we obtain $\lambda_*\chi^*\cE\cong \cO_{\Lambda(\p3)}(-\quantum)$. Taking into account such a latter isomorphism, Sequence \eqref{seqJumping} is exactly Sequence \eqref{seqJumpingFirst}.

In order to complete the proof, notice that the class of $J_\cE$ inside $H^*(\Lambda(\p3))$ is 
\begin{align*}
-c_2(\cG)&=c_2\left((K\otimes\cO_{\Lambda(\p3)}(1))^\vee\oplus (K\otimes\cS)\right)\\
&=c_2\left((K\otimes\cO_{\Lambda(\p3)}(1))^\vee\right)+c_2\left(K\otimes\cS\right)+c_1\left((K\otimes\cO_{\Lambda(\p3)}(1))^\vee\right)c_1\left(K\otimes\cS\right).
\end{align*}
By direct computation one checks that
\begin{gather*}
c_1\left((K\otimes\cO_{\Lambda(\p3)}(1))^\vee\right)=c_1\left(K\otimes\cS\right)=-\frac{k}{2}\sigma_1, \\
c_2\left((K\otimes\cO_{\Lambda(\p3)}(1))^\vee\right)=\frac{k(k-2)}{8} \sigma_1^2,\qquad 
c_2\left(K\otimes\cS\right)=\frac{k(k-2)}{8} \sigma_1^2 + \frac{k}{2}\sigma_2.
\end{gather*}
The statement then follows by combining the above equalities
\end{proof}

Now we turn our attention to the quadric $Q\subseteq\p4$. Recall that $\Lambda(Q)\cong\p3$. In this case $J_\cE$ is the support of the sheaf $R^1\lambda_*\chi^*\cE(-h)$ and it is a divisor (see Example \ref{eHulek}).

\begin{proposition}
Let $\cE$ be an even instanton bundle on $Q$ with quantum number $\quantum$. 

Then $J_{\cE}\subseteq\Lambda(Q)$ is a divisor of degree $\quantum$ endowed with a sheaf $\mathscr G$ fitting into
\begin{equation}
\label{seqJumpingQuadric}
0\longrightarrow H^\vee\otimes\cO_{\Lambda(Q)}(-2)\mapright{\Psi} H\otimes\cO_{\Lambda(Q)} \longrightarrow \mathscr G\longrightarrow 0.
\end{equation}
Moreover, $J_\cE$ is a hypersurface of degree $k$.
\end{proposition}
\begin{proof}
Splitting Monad \eqref{MonadQuadric} we obtain the exact sequences
\begin{equation}
\label{seqDisplayQuadric}
\begin{gathered}
0\longrightarrow \chi^*\cK(-h)\longrightarrow K\otimes\chi^*\cO_Q(-h)\longrightarrow H\otimes\chi^*\cO_Q\longrightarrow 0,\\
0\longrightarrow H^\vee\otimes\chi^*\cO_Q(-2h)\longrightarrow\chi^*\cK(-h)\longrightarrow \chi^*\cE(-h)\longrightarrow 0.
\end{gathered}
\end{equation}
As in the previous proof we apply the functor $\lambda_*$ to the two above sequences. In particular we have to identify the sheaves $R^i\lambda_*\chi^*\cO_Q(-th)$ for $t=0,1,2$ and $i\ge0$.

Recall that $\Lambda(Q)\cong\p3$ is endowed with the standard polarization $\cO_{\p3}(1)$ and $ \cO_{\mathscr L}$ fits into the exact sequence
\begin{equation}
\begin{aligned}
\label{seqUniversalQuadric}
0\longrightarrow \cO_{\Lambda(Q)}(-2)\boxtimes\cO_Q(-h)&\longrightarrow  \cO_{\Lambda(Q)}(-1)\boxtimes \cS\\
&\longrightarrow \cO_{\Lambda(Q)\times {Q}}\longrightarrow \cO_{\mathscr L}\longrightarrow0
\end{aligned}
\end{equation}
where $\cS$ is the unique spinor bundle on $Q$, whose non--zero sections vanish exactly on the lines inside $Q$. We then compute $R^i\lambda_*\chi^*\cO_Q(-th)$ by applying the functor $\widehat{\lambda}_*$ to Sequence \eqref{seqUniversalQuadric} tensored by $\widehat{\chi}^*\cO_Q(-th)$. We then deduce 
$$
R^1{\lambda}_*{\chi}^*\cO_Q(-2h)\cong\cO_{\Lambda(Q)}(-2),\qquad R^0{\lambda}_*{\chi}^*\cO_Q\cong\cO_{\Lambda(Q)},
$$
all the other $R^i$'s being zero. Replacing the above isomorphisms in Sequences \eqref{seqDisplayQuadric} we finally obtain Sequence \eqref{seqJumpingQuadric} where $\mathscr G\cong R^1\lambda_*\chi^*\cE(-h)$ and $\Psi$ is induced by the map $\psi\colon H^\vee\otimes\chi^*\cO_Q(-2h)\to K\otimes\chi^*\cO_Q(-h)$. 

Finally observe that $\cG$ is supported on the vanishing locus of $\det(\Psi)$. The statement then follows by noticing that $\Psi$ is represented by a ${k}/{2}$ square matrix of degree two homogenous forms (see Sequence \eqref{seqJumpingQuadric}).
\end{proof}

\bigskip
\noindent
Vincenzo Antonelli,\\
Dipartimento di Scienze Matematiche, Politecnico di Torino,\\
c.so Duca degli Abruzzi 24,\\
10129 Torino, Italy\\
e-mail: {\tt vincenzo.antonelli@polito.it}

\bigskip
\noindent
Gianfranco Casnati,\\
Dipartimento di Scienze Matematiche, Politecnico di Torino,\\
c.so Duca degli Abruzzi 24,\\
10129 Torino, Italy\\
e-mail: {\tt gianfranco.casnati@polito.it}

\bigskip
\noindent
Ozhan Genc,\\
Faculty of Mathematics and Computer Science, Jagiellonian University,\\
ul. {\L}ojasiewicza 6,\\
30-348 Krak{\'o}w, Poland\\
e-mail: {\tt ozhangenc@gmail.com}


\begin{thebibliography}{44}

\bibitem{A--K}
A.B. Altman, S.L. Kleiman: \emph{Compactifying the Picard scheme}. 
Adv. in Math. \textbf{35} (1980), 50--112.

\bibitem{A--O1} 
V. Ancona, G. Ottaviani: {\em Some applications of Beilinson's theorem to projective spaces and quadrics}. Forum Math. \textbf{3} (1991), 157--176.


\bibitem{A--C--G}
V. Antonelli, G. Casnati, O. Genc: {\em Instanton bundles on $\p1\times\bF^1$}. Commun. Alg. \textbf{49} (2021), 3594--3613.

\bibitem{A--M}
V. Antonelli, F. Malaspina: {\em Instanton bundles on the Segre threefold with Picard number three}. Math. Nachr. \textbf{293} (2020) 1026--1043. 

\bibitem{Ar}
  E. Arrondo: \emph{A home--made Hartshorne--Serre correspondence}. Comm. Algebra  \textbf{ 20} \rm (2007),  423--443.

\bibitem{A--C}
  E. Arrondo, L. Costa: \emph{Vector bundles on Fano $3$--folds without intermediate cohomology}. Comm. Algebra  \textbf{ 28} \rm (2000),  3899--3911.

\bibitem{Ba}
W. Barth: \emph{Some properties of stable rank-2 vector bundles on $\mathbb{P}n$}. Math. Ann. \textbf{226} (1977), 125--150.

\bibitem{B--E}
D. Bayer, D. Eisenbud: {\em Ribbons and their canonical embeddings}. Trans. Amer. Math. Soc. \textbf{347} (1995), 719--756. 

\bibitem{B--F1}
M.C. Brambilla, D. Faenzi: {\em Moduli spaces of rank-$2$ ACM bundles on prime Fano threefolds}. Michigan Math. J. \textbf{60} (2011), 113--148. 

\bibitem{B--F2}
M.C. Brambilla, D. Faenzi: {\em Vector bundles on Fano threefolds of genus $7$ and Brill--Noether loci}. Internat. J. Math. \textbf{25} (2014), 1450023, 59 pp..

\bibitem{B--M}
C. B\u anic\u a, N. Manolache: {\em Rank $2$ stable vector bundles on $\p3(\bC)$ with Chern classes $c_1=-1$, $c_2=4$}. Math. Z. \textbf{190} (1985), 315--339. 

\bibitem{B--P--S}
C. B\u anic\u a, M. Putinar, G. Schumacher: {\em Variation der globalen Ext in
Deformationen kompakter komplexer R\"aume}. Math. Ann. \textbf{250} (1980), 135--155.

\bibitem{Bea}
A. Beauville: {\em Vector bundles on Fano threefolds and $K3$ surfaces}. arXiv:1906.03594 [math.AG].



\bibitem{C--C--G--M}
G. Casnati, E. Coskun, O. Genc, F. Malaspina: \emph{Instanton bundles on the blow up of $\p3$ at a point}. arXiv:1909.10281 [math.AG],  to appear in Mich. Math. J., DOI: 10.1307/mmj/1601625614.


\bibitem{C--Fa--M2}
G. Casnati, D. Faenzi, F. Malaspina: {\em Rank two aCM bundles on the del Pezzo fourfold of degree $6$ and its general hyperplane section}. J. Pure Appl. Algebra \textbf{222} (2018), 585--609.

\bibitem{C--Fi--M}
G. Casnati, M. Filip, F. Malaspina: {\em Rank two aCM bundles on the del Pezzo threefold of degree $7$}. Rev. Mat. Complut. \textbf{30} (2017), 129--165.

\bibitem{Cs--Ge}
G. Casnati, O. Genc: \emph{Instanton bundles on two Fano threefolds of index $1$}. Forum Math. \textbf{32} (2020), 1315--1336.

\bibitem{C--G--N}
N. Chiarli, S. Greco, U. Nagel: \emph{On the genus and Hartshorne--Rao module of projective curves}. Math. Z. \textbf{229} (1998), 695--724.

\bibitem{Cl--Gr}
C.H. Clemens, Ph.A. Griffiths: \emph{The intermediate Jacobian of the cubic threefold}.  Ann. of Math. \textbf{95} (1972), 281--356

\bibitem{C--F}
I. Coand\u a, D. Faenzi: \emph{A refined stable restriction theorem for vector bundles on quadric threefolds}. Ann. Math. Pura Appl. (4) \textbf{193} (2014), 859--887.
 
\bibitem{C--H}
C. Ciliberto, J. Harris: {\em Surfaces of low degree containing a general canonical curve}. Comm. Algebra \textbf{27} (1999), 1127--1140.

\bibitem{C--MR}
L. Costa, R.M. Mir\'o--Roig: \emph{Monads and instanton bundles on smooth hyperquadrics}. Math. Nachr. \textbf{282} (2009), 169--179. 

\bibitem{El--Gr}
Ph. Ellia, L. Gruson: \emph{On the Buchsbaum index of rank two vector bundles on $\p3$}. Rend. Istit. Mat. Univ. Trieste, \textbf{47} (2015), 65--79.

\bibitem{El--St}
G. Ellingsrud, S.A. Str\o mme: \emph{Stable rank-2 vector bundles on $\p3$ with $c_1 = 0$ and $c_2 = 3$}. Math. Ann. \textbf{255} (1981), 123--138. 

\bibitem{Fa}
D. Faenzi: \emph{Even and odd instanton bundles on Fano threefolds of Picard number one}. Manuscripta Math. \textbf{144} (2014), 199--239.


  
\bibitem{G--H}
Ph. Griffiths, J. Harris: {\em Principles of algebraic geometry}. Wiley Classics Library, Wiley \rm(1994).

 
\bibitem{Ha1}
R. Hartshorne: {\em Stable reflexive sheaves}. Math. Ann. \textbf{254} (1980), 121--176.

\bibitem{Ha2}
R. Hartshorne: {\em Algebraic geometry}. G.T.M. 52, Springer \rm (1977).
 
\bibitem{Ha4}
R. Hartshorne: {\em Stable vector bundles of rank $2$ on $\p3$}. Math. Ann. \textbf{238} (1978), 229--280. 

\bibitem{Ha3}
R. Hartshorne: {\em Coherent functors}. Adv. Math. \textbf{140} (1998), 44--94.

\bibitem{H--S}
R. Hartshorne, I. Sols: {\em Stable rank $2$ vector bundles on $\p3$ with $c_1=-1$, $c_2=2$}. J. Reine Angew. Math. \textbf{325} (1981), 145--152.


\bibitem{Hu}
K. Hulek: {\em Stable rank-2 vector bundles on $\bP_2$ with $c_1$ odd}. Math. Ann. \textbf{242} (1979), 241--266.


\bibitem{H--L}
D. Huybrechts, M. Lehn: \emph {The geometry of moduli spaces of sheaves. Second edition}. Cambridge Mathematical Library, Cambridge U.P. \rm (2010).

\bibitem{I--P}
  V.A. Iskovskikh, Yu.G. Prokhorov: \emph{Fano varieties}. Algebraic Geometry V (A.N. Parshin and I.R. Shafarevich eds.), Encyclopedia of Mathematical Sciences  47, Springer, \rm (1999).

\bibitem{Isk2}
  V.A. Iskovskikh: \emph{Fano $3$--folds, II}. (Russian) Izv. Akad. Nauk SSSR Ser Mat \textbf{42} (1978), 504--549; translation in Math. USSR Izv. \textbf{12} (1978), 469-506.

\bibitem{J--M--P--S}
M.B. Jardim, G. Menet, D.M. Prata, H.N. S\'a Earp: {\em Holomorphic bundles for higher dimensional gauge theory}. Bull. London Math. Soc. \textbf{49} (2017), 117--132.

\bibitem{J--V}
M. Jardim, M. Verbitsky: {\em Trihyperk\"ahler reduction and instanton bundles on $\bC\p3$}. Compos. Math. \textbf{150} (2014), 1836--1868.


\bibitem{Kuz}
A. Kuznetsov: {\em Instanton bundles on Fano threefolds}. Cent. Eur. J. Math. \textbf{10} (2012), 1198--1231.

\bibitem{K--P--S}
A.G. Kuznetsov, Y.G. Prokhorov, C.A. Shramov: {\em Hilbert schemes of lines and conics and automorphism groups of Fano threefolds}. Japan. J. Math. \textbf{13} (2018), 109--185.

\bibitem{Mad}
C. Madonna: {\em A splitting criterion for rank $2$ vector bundles on hypersurfaces in $\p4$}. Rend. Sem. Mat. Univ. Pol. Torino \textbf{56} (1998), 43--54.



\bibitem{M--M--PL}
F. Malaspina, S. Marchesi, J. Pons--Llopis: {\em Instanton bundles on the flag variety $F(0,1,2)$}. Ann. Sc. Norm. Super. Pisa Cl. Sci. (5) \textbf{20} (2020), 1469--1505.



\bibitem{Ma1}
M. Maruyama: {\em Boundedness of semistable sheaves of small ranks}. Nagoya Math. J. \textbf{78} (1980), 65--94.

\bibitem{Ma2}
M. Maruyama: {\em Moduli of stable sheaves. I}. J. Math. Kyoto Univ. \textbf{17} (1977), 91--126.

\bibitem{Ma3}
M. Maruyama: {\em Moduli of stable sheaves. II}. J. Math. Kyoto Univ. \textbf{18} (1978), 557--614.

\bibitem{N--N--S1}
U. Nagel, R. Notari, M.L. Spreafico: {\em Curves of degree two and ropes on a line: their ideals and even liaison classes}. J. Algebra \textbf{265} (2003), 772--793.



\bibitem{O--S--S}
  C. Okonek, M. Schneider, H. Spindler: {\em Vector bundles on complex projective spaces}. Progress in Mathematics 3, \rm(1980).


\bibitem{Ott2}
G. Ottaviani: {\em Spinor bundles on quadrics}. Trans. Am. Math. Soc. \textbf{307} (1988), 301--316.

\bibitem{O--S}
G. Ottaviani, M. Szurek: {\em On moduli of stable $2$-bundles with small Chern classes on $Q_3$. With an appendix by Nicolae Manolache}. Ann. Mat. Pura Appl. \textbf{167} (1994), 191--241.

\bibitem{S--W}
M. Szurek, J.A. Wiśniewski: \emph{Conics, conic fibrations and stable vector bundles of rank $2$ on some Fano threefolds}. Rev. Roumaine Math. Pures Appl. \textbf{38} (1993), 729--741. 

  
\bibitem{Tik1}
A. S. Tikhomirov: {\em Moduli of mathematical instanton vector bundles with odd $c_2$ on projective space}. (Russian) Izv. Ross. Akad. Nauk Ser. Mat. \textbf{76} (2012), 143--224; translation in Izv. Math. \textbf{76} (2012), 991--1073.

\bibitem{Tik2}
A.S. Tikhomirov: {\em Moduli of mathematical instanton vector bundles with even $c_2$ on projective space}. (Russian) Izv. Ross. Akad. Nauk Ser. Mat. \textbf{77} (2013), 139--168; translation in Izv. Math. \textbf{77} (2013), 1195--1223.

\bibitem{Tik}
S.A. Tikhomirov: {\em Families of stable bundles of rank $2$ with $c_1=-1$ on the space $\p3$}. (Russian)
Sibirsk. Mat. Zh. \textbf{55} (2014), 1396--1403; translation in
Sib. Math. J. \textbf{55} (2014), 1137--1143. 


\end{thebibliography}
\end{document}